\documentclass[11pt]{article}
\usepackage[english]{babel}

\usepackage 
			[margin=10pt,
			font=small,
			labelfont=bf,
			labelsep=endash, 
			justification=raggedright, 
			singlelinecheck=false] 
			{caption}

\usepackage[notref,notcite]{}
\usepackage{todonotes}
\usepackage{calc}
\usepackage[numbers]{natbib}

\usepackage{
	amsopn,
	amsmath,
	amsfonts,
	amssymb,
	color,
	comment,
	enumerate,
	enumitem,
	dsfont,
	graphicx,
	latexsym,
	stmaryrd,
	setspace,
	tikz
	}
	
	\usepackage{xifthen}

\usepackage{mathtools}
\mathtoolsset{showonlyrefs}

\usetikzlibrary{decorations.pathmorphing,matrix}

\specialcomment{percom}
	{{\bf Personal Comment, will be deleted later on: }\bgroup \it}{\egroup}

\usepackage{amsthm}
	\newtheorem{theo}{Theorem}[section]

	\newtheorem{lem}[theo]{Lemma}
	\newtheorem{prop}[theo]{Proposition}
	\newtheorem{remark}[theo]{Remark}
	
	\theoremstyle{definition}

	\newtheoremstyle{proof}
		{3pt}
		{3pt}
		{}
		{}
		{\itshape}
		{:}
		{.5em}
		{}

\newcommand{\Z}{\ensuremath{\mathbb{Z}} }
\newcommand{\N}{\ensuremath{\mathbb{N}} }

\newcommand{\R}{\ensuremath{\mathbb{R}} }

\newcommand{\E}{\ensuremath{\mathbb{E}} }
\renewcommand{\P}{\ensuremath{\mathbb{P}} }
\newcommand{\1}{\mathds 1 }

\newcommand{\skorcon}{ \xrightarrow{\tiny \mathcal L_{M_1}-s}}

\newcommand{\eps}{\varepsilon}

\newcommand{\ol}{\overline}

\newcommand{\wt}{\widetilde}

\newcommand{\din}{\Delta_{i,k}^n}

\newcommand{\diff}{\; d}
\newcommand{\eqdist}{\overset{\text{d}}{=}}
\newcommand{\eqspace}{\hspace{4em}}

\newcommand{\PP}{\mathbb{P}}
\newcommand{\D}{\mathbb{D}}

\newcommand{\F}{\mathcal{F}}

\newcommand{\f}{\mathcal{F}}

\newcommand{\laa}{\mathcal{L}}

\newcommand{\schw}{\stackrel{d}{\longrightarrow}}
\newcommand{\eqschw}{\stackrel{d}{=}}
\newcommand{\toop}{\stackrel{\PP}{\longrightarrow}}

\newcommand{\toas}{\stackrel{\mbox{\tiny a.s.}}{\longrightarrow}}

\newcommand{\fidi}{\stackrel{f.i.d.i.}{\longrightarrow}}

\newcommand{\tol}[1][]{%
\ifthenelse{\equal{#1}{}}{\stackrel{\laa}{\longrightarrow}}{\stackrel{L^{#1}}{\longrightarrow}}%
}

\newcommand{\tols}{\stackrel{\laa-\mbox{\tiny s}}{\longrightarrow}}

\newcommand{\ucp}{\xrightarrow{\mbox{\tiny u.c.p.}}}

\newcommand{\bee}{\begin{equation}}
\newcommand{\eee}{\end{equation}}
\newcommand{\bea}{\begin{eqnarray}}
\newcommand{\eea}{\end{eqnarray}}
\newcommand{\bean}{\begin{eqnarray*}}
\newcommand{\eean}{\end{eqnarray*}}

\newcommand{\sbs}{\ensuremath{S\beta S}}
\newcommand{\sbssp}[1]{S$\beta$S(\ensuremath{#1})}
\newcommand{\G}{\mathcal F}
\newcommand{\Cov}{\text{cov}}
\newcommand{\var}{\text{var}}

\begin{document}

\title{On limit theory for functionals of \\ stationary increments L\'evy driven moving averages}
\author{Andreas Basse-O'Connor\thanks{Department
of Mathematics, Aarhus University, 
E-mail: basse@math.au.dk.}  \and
Claudio Heinrich\thanks{Department
of Mathematics, Aarhus University \& Norwegian Computing Center Oslo \newline
\hspace*{1.5em} E-mail: claudio.heinrich@nr.no}  \and
Mark Podolskij\thanks{Department
of Mathematics, Aarhus University,
E-mail: mpodolskij@math.au.dk.}}

\maketitle

\begin{abstract}
In this paper we obtain  new limit theorems for variational  functionals of high frequency observations of stationary increments L\'evy driven moving averages. We will see that 
the asymptotic behaviour of such   functionals heavily depends on the kernel, the driving L\'evy process and the properties of the functional  under consideration. We show the ``law of large numbers'' for our class of statistics, which consists of three different cases. For one of the appearing limits, which we refer to as the ergodic type limit, we also prove the associated weak limit theory, which again consists of three different cases. Our work is related to  \cite{BasLacPod2016, BasPod2017}, who considered power variation functionals of stationary increments L\'evy driven moving averages. 

\end{abstract}

\bigskip


\section{Introduction}
\label{secMainResults}
 \label{secInt}
\setcounter{equation}{0}
\renewcommand{\theequation}{\thesection.\arabic{equation}}

The last two decades have witnessed a great progress in limit theory for high frequency functionals of continuous time stochastic processes. The interest in infill asymptotics has been motivated by the increasing availability of high frequency data in natural and social sciences such as finance, physics, biology or medicine.  Limit theorems in the high frequency framework are an important probabilistic tool for the analysis of small scale fluctuations of the underlying stochastic process and have numerous applications in mathematical statistics e.g.\ in the field of parametric estimation and testing.  Such limit theory has been investigated in various model classes including 
It\^o semimartingales (see e.g.\ \cite{BarGraJacPodShe2006,  Jac2008, JacPro2012}), (multi)fractional Brownian motion and related processes (see e.g.\ \cite{BarSur2013, BarCorPod2009, BarCorPod2011, BarCorPodWoe2009, GuyLeo1989, LebPod2017}), and many others.  

In this paper we investigate the asymptotic theory for high frequency functionals of stationary increments L\'evy driven moving averages. More specifically, we focus on 
an infinitely divisible process with stationary increments $(X_t)_{t\geq 0}$, defined on a  probability space $(\Omega, \mathcal F,   \mathbb P)$, given  as
\begin{equation} \label{def-of-X-43}
X_t= \int_{-\infty}^t \big\{g(t-s)-g_0(-s)\big\}\, dL_s,
\end{equation} 
where $L=(L_t)_{t\in \R}$ is a two-sided L\'evy process with no Gaussian component and $L_0=0$, and
$g, g_0:\R \to \R$ are continuous functions  vanishing on $(-\infty,0)$. In particular, this class of stochastic processes contains the linear fractional stable motion, which has the form \eqref{def-of-X-43} with $g(s)=g_0(s)=s_+^\alpha$ and the driving L\'evy process $L$ is symmetric stable. The linear fractional stable motion  is the most common heavy-tailed self-similar process, and hence exhibit both the Joseph and Noah effects of Mandelbrot, cf.\ \cite[Chapter~7]{SamTaq1994}.  
Fractional L\'evy processes are other examples of processes of the form \eqref{def-of-X-43}, see e.g.\ \cite[Chapter~2.6.8]{pip-taqqu}. 
Recent  papers  address various topics on  linear fractional stable motions including analysis of semimartingale property \cite{BasRos2016}, fine scale behavior \cite{BenCohIst2004,Gla2015}, simulation techniques \cite{CohLacLed2008} and statistical inference \cite{AyaHam2012, DanIst2015, MazOtrPod2017, PipTaqAbr2007}.
We consider the class of variational functionals of the type
\begin{align} \label{vn}
V(f;k)^n := a_n\sum_{i=k}^{n} f(b_n\Delta_{i,k}^{n} X),
\end{align}
where $f:\R \to \R$ is a measurable function, 
$(a_n)_{n\in\N}, (b_n)_{n\in\N}$ are suitable normalising sequences, and 
the operator $\Delta_{i,k}^{n} X$  denotes the 
$k$th order increments of $X$  defined as 
\begin{align} \label{filter}
\Delta_{i,k}^{n} X:= \sum_{j=0}^k (-1)^j \binom{k}{j} X_{(i-j)/n}, \qquad i\geq k.
\end{align}
The usual first and second order increments take the forms $\Delta^n_{i,1}X = X_{i/n}-X_{(i-1)/n}$ and $\Delta^n_{i,2}X= X_{i/n}-2X_{(i-1)/n}+X_{(i-2)/n}$. The reason for considering general $k$th order increments lies in statistical applications. Indeed,
using higher order increments, with $k\geq 2$, is often  desirable since this  gives rise to better convergence rates for various  estimators (cf. \cite{MazOtrPod2017}). This fact is also  seen in our asymptotic results Theorems~\ref{th1}, \ref{lsjdlfjssss} and \ref{the2ndGenVar}. 
The choice of the normalising sequences $(a_n)_{n\in\N}$ and $(b_n)_{n\in\N}$ depends on the interplay between the form of the kernel $g$, the infinitesimal properties of the driving L\'evy process $L$ 
and the growth/smoothness of the function $f$. 

The asymptotic behaviour of statistics of the form \eqref{vn} in the context of power variation, i.e. $f(x)=|x|^p$ for some $p>0$, has been characterized in the work \cite{BasLacPod2016, BasPod2017}. Further papers on related topics
include \cite{PipTaq2003} that investigate asymptotic normality for functionals of the type \eqref{vn} in the low frequency setting and for \emph{bounded} functions $f$ (the article \cite{PipTaqAbr2007} extends the results of \cite{PipTaq2003} to certain unbounded functions). Much more is known about weak limit theory for statistics of \emph{discrete} moving averages driven by heavy tailed i.i.d.\ noise; we refer to \cite{HoHsi1997, Sur2002, Sur2004} among others. However, the asymptotic theory is investigated mostly for bounded functions $f$ and under assumptions on the kernel and the noise process, which are not  comparable to ours. We will conclude the  discussion of related literature by mentioning the two papers  \cite[Section~5]{BenCohIst2004} and  \cite{Gla2015}, which  show ``law of large numbers'' results of the ergodic type in the context of fractional L\'evy processes. 

The aim of this work is to investigate the limit theorems for general functionals $V(f;k)^n$. We will start with first order asymptotic results, which consist of three different limits depending on the interplay 
between $f$, $g$ and $L$. More specifically, the ``laws of large numbers" include stable convergence
towards a certain random variable, ergodic type convergence to a constant when the driving motion
$L$ is assumed to be symmetric $\beta$-stable and convergence in probability to an integral of some stochastic process. In the second step we will also prove three weak limit theorems associated with the ergodic type convergence, consisting of a central limit theorem and two convergence results towards stable distributions. 
Motivated by statistical applications, such as  parametric estimation of linear fractional stable motion (cf. \cite{AyaHam2012, DanIst2015, MazOtrPod2017, PipTaqAbr2007}), we will apply our theory to  functions $f$ of the form
\begin{align} \label{functions}
& f_1(x)=|x|^p, \quad p>0 & \text{(power variation)} \nonumber \\
& f_2(x)=|x|^{-p}\1_{\{x\neq 0\}}, \quad p\in (0,1) & \text{(negative power variation)} \nonumber \\
& f_3(x)= \cos(ux) \text{ or } \sin(ux)  & \text{(empirical characteristic function)} \\
& f_4(x)= \mathds 1_{(- \infty, u]}(x)  & \text{(empirical distribution function)} \nonumber \\
& f_5(x)= \log(|x|) \1_{\{x\neq 0\}} & \text{($\log$-variation)} \nonumber
\end{align}    
among others. One of the major difficulties when showing weak limit theorems lies in the fact that 
the ideas suggested in e.g. \cite{HoHsi1997, PipTaq2003, Sur2002, Sur2004} in the setting of bounded
functions $f$ do not directly extend to a more general class of functions (also the proofs in \cite{BasLacPod2016} for the power variation case use the specific form of the function $ f(x)=|x|^p$).
As it has been noticed in earlier papers on discrete moving averages 
(see e.g.\ \cite{Sur2002, Sur2004} and references therein) the \emph{Appell rank} of the function $f$ often plays an important role for the weak limit theory. It is defined as $m^{\star}=\min\{m \in \N:~
\Phi_{\rho}^{(m)}(0)\neq 0\} $  with 
\[
\Phi_{\rho}(x):= \E[f(x+\rho S)]-\E[f(\rho S)],
\] 
where $S$ is a symmetric $\beta$-stable random variable with scale parameter $1$,  $\rho>0$  and $\Phi_{\rho}^{(m)}$ denotes the $m$th derivative of $x\mapsto \Phi_{\rho}(x)$.  In this paper we will show that it is much more convenient to impose assumptions on the function $\Phi$, rather than on the function $f$ itself, to obtain weak limit theorems for a wide class of functionals  $V(f;k)^n$. This is one of the main results of our work.

The paper is structured as follows. Section \ref{sec2} presents the required assumptions, the main  results and some remarks and examples. We present  some preliminaries
in Section~\ref{sec2a}. 
The proofs of the first order asymptotic 
results are collected in Section~\ref{secPro1}.  Section~\ref{sec4} is devoted to the proofs of  weak limit theorems, with a few more technical results postponed to Section~\ref{lsjdfghsdfoghs}.

\section{The setting and main results}
\label{sec2}
\setcounter{equation}{0}
\renewcommand{\theequation}{\thesection.\arabic{equation}}

We start by introducing various definitions, notations and assumptions that will be important for the presentation of the main results. We recall that the \textit{Blumenthal--Getoor index} of $L$ is defined as
\begin{align} \label{def-B-G}
\beta:=\inf\Big\{r\geq 0: \int_{-1}^1 |x|^r\,\nu(dx)<\infty\Big\}\in [0,2],
\end{align} 
where $\nu$ denotes the L\'evy measure of $L$. Furthermore,  $\Delta L_s:=L_s-L_{s-}$ with $L_{s-}:=\lim_{u\uparrow s,\,u<s} L_u$ stands for the jump size of $L$ at point $s$. If $L$ is stable with index of stability $\beta\in(0,2)$, the index of stability and the Blumenthal-Getoor index coincide, and both will be denoted by $\beta$. Let $\mathbb F=(\mathcal F_t)_{t\in\R}$ be the filtration generated by the L\'evy process $L$ and $(T_m)_{m\geq 1}$  be a sequence of $\mathbb F$-stopping times  that exhausts the jumps of $(L_t)_{t\geq 0}$. That is,  $ \{T_m(\omega):m\geq 1\} = \{t\geq 0: \Delta L_t(\omega)\neq 0\}$ and  $T_m(\omega)\neq T_n(\omega)$ for all $m\neq n$ with $T_m(\omega)<\infty$.

Our first set of conditions, which has been originally introduced in \cite{BasLacPod2016}, concerns the behaviour of the L\'evy measure $\nu$ at infinity and the functional form of the kernel $g$: \\ \\
 \textbf{Assumption~(A):} {\it
The function $g\!:\R\to\R$ satisfies 
\begin{equation}\label{kshs}
g(t)\sim  t^\alpha\qquad \text{as } t\downarrow 0\quad \text{for some }\alpha>0, 
\end{equation}
where $g(t)\sim w(t)$ as $t\downarrow 0$ means that $\lim_{t\downarrow 0}g(t)/w(t)= 1$. 
For some $\theta\in (0,2]$ it holds that $\limsup_{t\to \infty} \nu(x\!:|x|\geq t) t^{\theta}<\infty$ and $g - g_0$ is a bounded function in $L^{\theta}(\R_+)$.
 Furthermore,  $g \in C^k((0, \infty))$  and there  exists a $\delta>0$ such that  $|g^{(k)}(t)|\leq C t^{\alpha-k}$ for all $t\in (0,\delta)$, and such that both $|g'|$ and $|g^{(k)}|$ are in $L^\theta((\delta,\infty))$, and  are decreasing on $(\delta,\infty)$. }\\ \\
Assumption (A) ensures in particular that the process $X_t$, introduced in 
\eqref{def-of-X-43},  is well-defined in the sense of \cite{RajRos1989}, see\ \cite[Section~2.4]{BasLacPod2016}.  
When $L$ is a $\beta$-stable L\'evy process, we may and do  choose $\theta = \beta$.
By adjusting the L\'evy measure $\nu$, we may also include the case where \eqref{kshs} is replaced by  $g(t)\sim  c_0 t^\alpha$ as $t\downarrow 0$ for some $c_0\neq  0$.

For Theorem~\ref{th1}(i) below, we need to slightly strengthen Assumption (A)  if $\theta=1$: \\ \\
\textbf{Assumption~(A-log):} In addition to  (A) suppose that 
\begin{equation} 
\int_\delta^\infty |g^{(k)}(s)|^\theta \log(1/|g^{(k)}(s)|)\,ds<\infty.
\end{equation}
In order to formulate our main results, we require some more notation. 
For $p>0$ we denote by $C^p(\R)$ the space of $r:=[p]$-times continuous differentiable functions $f:\R\to\R$ such that $f^{(r)}$ is locally $(p-r)$-H\"older continuous if $p\not\in\N$.
We introduce the function $h_k\!:\R\to\R$  by 
\begin{align}\label{def-h-13}
h_k(x)&:=  \sum_{j=0}^k (-1)^j \binom{k}{j} (x-j)_{+}^{\alpha},\qquad x\in \R,\vspace{-0.75 em}
\end{align} 
where  $y_+:=\max\{y,0\}$ for all $y\in \R$. We recall that a sequence $(Z^n)_{n\in\N}$ of random variables defined on $(\Omega,\mathcal F, \mathbb{P})$ with values in a Polish space 
$(E,\mathcal E)$ converges stably in law to $Z$, which is defined on an extension $(\Omega',\mathcal F', \mathbb{P}')$ of the original probability space, if for all bounded continuous $g:E\to\R$ and for all bounded $\mathcal F$-measurable random variables $Y$ it holds that 
\[\E[g(Z^n)Y]\to\E'[g(Z)Y], \] 
where $\E'$ denotes the expectation on the extended space. We denote the stable 
convergence in law by $Z^n\tols Z$, and refer to \cite{HaeLus2015, Ren1963} for more details. Note, in particular, that stable convergence in law is a stronger property than convergence in law, but a weaker property than convergence in probability. In the framework of stochastic processes we write $Z^n \ucp Z$ for uniform convergence in probability, i.e.\ when  
$\sup_{t \in [0,T]} |Z_t^n - Z_t| \toop 0$ holds for all $T>0$. Furthermore, we denote by $Z^n \fidi Z$
the stable convergence of finite dimensional distributions.

\subsection{Law of large numbers} \label{sec2.1}

Our first theorem presents the ``law of large numbers" for the statistic $V(f;k)^n$ defined at \eqref{vn}. 
The sequence $(U_m)_{m\geq 1}$ below is i.i.d.\ $\mathcal U(0,1)$-distributed,  defined on an extension $(\Omega', \mathcal F', \mathbb P')$ and independent  of $\f$. Here and throughout the paper we denote by $S\beta S(\rho)$ the symmetric $\beta$-stable distribution with scale parameter 
$\rho>0$, that is $Y\sim S\beta S(\rho)$ if $\E[ \exp(i \theta Y)]= \exp(-|\rho \theta|^\beta)$ for all $\theta\in \R$.

\begin{theo} \label{th1}
Suppose Assumption (A) holds and  assume that the Blumenthal--Getoor index   satisfies $\beta<2$. 
We have  the following three cases:
\begin{itemize}
\item[(i)] Let $k>\alpha$ and suppose that (A-log) holds if  $\theta=1$. For some $p>\beta\vee\frac 1{k-\alpha}$ assume that $f\in C^p(\R)$ and $f^{(j)}(0)=0$ for $j=0,\dots, [p]$.  
With the normalising sequences $a_n=1$ and $b_n=n^{\alpha}$ we obtain the stable convergence 
\[
V(f;k)^n \tols \sum_{m:\,T_m\in [0,1]} \sum_{l=0}^\infty f\big(\Delta L_{T_m} h_k(l+U_m)\big).\]
\item[(ii)] Suppose that  $L$ is a symmetric $\beta$-stable L\'evy process with scale parameter 
$\rho_L>0$. Moreover, assume that  $\E[|f(L_1)|]<\infty$, and   $H:=\alpha + 1/\beta < k$. Then, setting $a_n=1/n$
and $b_n=n^H$, we obtain
\begin{align} \label{part22}
V(f;k)^n \toop  \E[f(\rho_0 S)],
\end{align}
where $S \sim S\beta S(1)$ and $\rho_0=\rho_L\|h_k\|_{L^\beta(\R)}$. 

\item[(iii)] Suppose that $(1\vee\beta)(k-\alpha)<1$ and  that $f$ is continuous and satisfies 
$|f(x)|\leq C(1\vee |x|^q)$ for all $x\in \R$,
for some $q, C>0$ with $q(k-\alpha)<1$. With the normalising sequences $a_n=1/n$
and $b_n=n^k$ it holds that 
\begin{equation} \label{part32}
 V(f;k)^n \toop \int_0^1 f(F_u)\, du 
 \end{equation}
where $(F_u)_{u\in \R}$ is  defined by 
\begin{equation}\label{FDefGenVar}
F_u=  \int_{-\infty}^u g^{(k)}(u-s) \,dL_s\quad \text{a.s.\ for all }u\in \R.
 \end{equation} 
\end{itemize}
\end{theo} 

Theorem~\ref{th1} may be viewed as a generalization of \cite[Theorem 1.1]{BasLacPod2016} from power variation to general functionals.
The limiting random variable in Theorem~\ref{th1}(i) is indeed well-defined, as we show in Lemma~\ref{LimEx} below. 
We remark that one of the conditions of Theorem~\ref{th1}(i) is the restriction $\alpha<k-1/p$. This restriction on the parameter 
$\alpha$ gets weaker when $p$ is getting larger, but on the other hand the condition $f\in C^p(\R)$ is stronger for a larger 
$p$. Thus, there is a trade-off between these two conditions. 

The three cases of the theorem are closely related to the three limits for the 
power variation derived in \cite[Theorem 1.1]{BasLacPod2016}. Let us briefly explain the main intuition behind Theorems \ref{th1}(ii) and (iii).

The crucial step in the proof of Theorem \ref{th1}(ii) is the approximation 
\begin{align} \label{approximation}
\Delta_{i,k}^{n} X \approx \Delta_{i,k}^{n} Y \qquad \text{in probability}
\end{align}
where $Y$ is the  linear fractional stable motion defined via
\begin{align} \label{flm}
Y_t :=  \int_{\R} \{(t-s)_{+}^\alpha  - ( -s)_{+}^\alpha \}\, dL_s.
\end{align} 
It is well known that the process $Y$ 
is $H$-self-similar and its increment process is  ergodic  (see e.g.\ \cite{CamHarWer1987}). Hence, under assumptions
of Theorem \ref{th1}(ii), we may conclude by Birkhoff's ergodic theorem for e.g.\ $k=1$:
\begin{align*}
V(f;1)^n \approx \frac{1}{n} \sum_{i=1}^{n} f \left(n^H(Y_{i/n} - Y_{(i-1)/n}) \right) 
\eqschw \frac{1}{n} \sum_{i=1}^{n} f \left(Y_i - Y_{i-1} \right) \toas \E[f(Y_1-Y_0)]. 
\end{align*}
This is exactly the statement of \eqref{part22} for the case $k=1$. 

Under assumptions of Theorem \ref{th1}(iii) it turns out that the stochastic process $F$ defined at 
\eqref{FDefGenVar} is a version of the $k$th derivative of $X$. Hence, we conclude by Taylor expansion:
\[
V(f;k)^n=\frac 1 n \sum_{i=k}^n f(n^k\din X)\approx \frac 1 n \sum_{i=k}^n f\left(F_{(i-1)/n}
\right)\toop\int_0^1 f(F_u)\,du,\quad \text{as }n\to\infty.
\]
This explains the statement of Theorem \ref{th1}(iii).

\begin{remark} \label{rem1} \rm
In contrast to the power variation case investigated in  \cite{BasLacPod2016}, the assumptions of Theorems~\ref{th1}(i) and (ii), and of Theorems~\ref{th1}(i) and (iii), are not mutually exclusive, and hence two limit theorems can hold at the same time. This phenomenon appears already in the simpler setting of L\'evy processes. Assume for example that $L$ is a symmetric $\beta$-stable L\'evy process and consider the function 
$f(x)=\sin^2(x)$.  If $k=1$ and we choose $a_n=b_n=1$ we deduce the convergence
\[
\sum_{i=1}^{n} \sin^2(\Delta_{i,1}^{n} L) \toas \sum_{m:\,T_m\in[0,1]}\sin^2(\Delta L_{T_m}) < \infty,
\]
using, in particular,  $|f(x)| \leq Cx^2$. On the other hand when we choose the normalising sequences $a_n=n^{-1}$ and $b_n=n^{1/\beta}$ we readily deduce by strong law of large numbers that 
\[
\frac 1 n \sum_{i=1}^{n} \sin^2(n^{1/\beta}\Delta_{i,1}^{n} L)\ \eqdist\  \frac 1 n 
\sum_{i=1}^{n} \sin^2(L_i - L_{i-1}) \toas \E[\sin^2(L_1)].
\]
This example shows that we can obtain two different limits for two different scalings. \qed
\end{remark}

In the next step we present a functional version of Theorem~\ref{th1}. For this purpose we introduce 
the sequence of processes 
\begin{align*} 
V(f;k)^n_t := a_n\sum_{i=k}^{[nt]} f(b_n\Delta_{i,k}^{n} X).
\end{align*}
In the proposition below we will use the Skorokhod $M_1$-topology, which was introduced in \cite{Sko1956}.  For a  detailed exposition  we refer to \cite{Whi2002}.

\begin{prop} \label{prop1} 
Suppose Assumption (A) holds and  assume that the Blumenthal--Getoor index   satisfies $\beta<2$. 
We have  the following three cases:
\begin{itemize}
\item[(i)] Under conditions of Theorem \ref{th1}(i) we obtain the stable convergence
\[
V(f;k)^n_t \fidi V(f;k)_t:=\sum_{m:T_m\in [0,t]} \sum_{l=0}^\infty f\big( \Delta L_{T_m} h_k(l+U_m)\big).
\]
Moreover, the stable convergence also holds 
with respect to  Skorokhod $M_1$-topology if additionally the following assumption is satisfied:
\begin{enumerate}[label=(FC)]\setcounter{enumi}{1}
\item \label{con3} 
Each of the two functions $x\mapsto f(x) \mathds 1_{\{x\geq 0\}}$ and $x\mapsto f(x) \mathds1_{\{x< 0\}}$ is either non-negative or non-positive.
\end{enumerate}
\item[(ii)] Under conditions of Theorem \ref{th1}(ii) we deduce that 
\begin{align*} 
V(f;k)^n_t \ucp  t\E[f(\rho_0 S)],
\end{align*}
where $S$ and $\rho_0$ have been introduced in \eqref{part22}. 

\item[(iii)] Under conditions of Theorem \ref{th1}(iii) we have 
\begin{equation*} 
 V(f;k)^n_t \ucp \int_0^t f(F_u)\, du 
 \end{equation*}
where $(F_u)_{u\in \R}$ has been defined at  \eqref{FDefGenVar}.
\end{itemize}
\end{prop}

We remark that the uniform convergence results of Proposition~\ref{prop1}(ii) and (iii) are easily obtained
from Theorem~\ref{th1}(ii) and (iii) by the following argument. Observe the decomposition $f=f^{+}
-f^{-}$, where $f^{+}$ (resp. $f^{-}$) denotes the positive (resp.\ negative) part of $f$. Then $f^{+},
f^{-}$ satisfy the same assumptions as $f$ in the setting of Theorem~\ref{th1}(ii) and (iii). 
Furthermore, since $f^{+}, f^{-}\geq 0$, the statistics $V(f^{+};k)^n_t$ and $V(f^{-};k)^n_t$ are increasing in $t$ and the corresponding limits in Proposition \ref{prop1}(ii) and (iii) are continuous in 
$t$. Consequently, the uniform convergence is obtained from the  pointwise convergence by Dini's theorem.

\subsection{Weak limit theorems} \label{sec2.2}

In this section we present  weak limit theorems associated to the ergodic type limit from 
Theorem~\ref{th1}(ii). Throughout this section we assume that $\E[|f(S)|]<\infty$, where  $S \sim S\beta S(1)$. As  mentioned in the introduction, the crucial quantity in this context is the function 
$\Phi_{\rho}$ defined via
\begin{align} \label{Phi}
\Phi_\rho(x)= \E[f(x+\rho S)]-\E[f(\rho S)],\qquad x\in \R, \,\rho>0.
\end{align}
 Similarly to limit theory for discrete moving averages, see e.g.\ 
 \cite{HoHsi1997,Sur2002, Sur2004}, the Appell rank of the function $f$ often plays a key role for  
the asymptotic behaviour of the statistic $V(f;k)^n- \E[f(\rho_0 S)]$. In our setting, the Appell rank $m^\star_\rho$  is defined as 
\[
m^{\star}_\rho:=\min\{r\in \N\,:\, \Phi^{(r)}_\rho(0)\neq 0\},
\]
where $\Phi^{(r)}_\rho(x):=\frac{\partial^r}{\partial x^r} \Phi_\rho(x)$ for $r=1,2,\dots$. Note that we have Appell rank one if and only if 
$\Phi'_\rho(0)\neq 0$, and Appell rank greater or equal two if and only if $\Phi_\rho'(0)=0$.  The Appell rank is an analogue of the Hermite rank  used in the context of Gaussian processes. However, the non-Gaussian case is  usually much more complicated due to the lack of orthogonal series expansions. While the Appell rank $m^{\star}_\rho$ usually depends on the parameter 
$\rho$, we always have that $m^{\star}_\rho= 1$ for all $\rho>0$ in  the framework of the imaginary part of the characteristic function $f_3(x)=\sin(ux)$ and the empirical distribution function $f_4$ (cf.\ Remark~\ref{lsdjflsdhfh}). Moreover,  $m^{\star}_\rho >1$ for all $\rho>0$ when   $f$ is an even function, in fact, in this case we have that $0=\frac{\partial }{\partial x} \Phi_\rho(0)=\frac{\partial^2 }{\partial x\partial \rho} \Phi_\rho(0)$ (cf.\ Remark~\ref{lsdjflsdhfh}). 
Indeed, $m^{\star}_\rho>1$ for all $\rho>0$ therefore holds in the setting of power variations  $f_1$ and $f_2$, real part of the characteristic function $f_3(x)=\cos(ux)$ and the log-variation $f_5$. 

For our weak limit theorems we will need the following smoothness assumptions on $\Phi_\rho$:

\noindent
\textbf{Assumption (B):} {\it
The function $(\rho,x)\mapsto \Phi_\rho(x)$ is   $C^{1,2}((0,\infty)\times \R)$, and 
for all $\eps\in (0,1)$ there are    $p\in [0,1]$ and $C>0$ such that,   for all 
$\rho\in[\eps,\eps^{-1}]$ and $x, y\in \R$  
\begin{align}\label{Est2}
&\Big|\Phi_\rho(x)-\Phi_\rho(y)\Big|\leq C|x-y|^p,
  \\ 
& 
\Big|\frac{\partial^{j+r} }{\partial x^j\partial \rho^r}\Phi_\rho(x)\Big|\leq C \qquad \text{for all }j=0, 1, 2\text{ and }r=0,1\textrm{ with } r+j>0. \label{Est3}
\end{align}
}


Note that \eqref{Est3} implies Lipschitz continuity of $\Phi_\rho$, and therefore the $p$-H\"older assumption
\eqref{Est2} may be viewed as a growth condition on $\Phi_\rho$. In particular, \eqref{Est2} implies that $|\Phi_\rho(x)|\leq C|x|^p$. Note also that \eqref{Est3} implies  \eqref{Est2} with $p=1$, however, in several cases we need $p<1$.   Before presenting our main weak limit theorems, we  remark that 
Assumption~(B) is satisfied for our key examples, its proof is  postponed to the end of Section~\ref{lsjdfghsdfoghs}.

\begin{remark} \rm \label{examp} The following two classes of functions $f$ satisfies (B). 
\begin{enumerate}[wide, labelwidth=!, labelindent=0pt, label=(\roman*)]
\item (\textit{Bounded functions}). Any bounded measurable function $f$ satisfies (B) for any $p\in [0,1]$. 
This covers, in particular,  the empirical distribution function $f_4 (x)= \1_{(-\infty,u]}(x)$, and  the empirical characteristic functions $f_3(x)=\sin(u x)$ or $f_3=\cos(ux)$ from \eqref{functions}, where   $u\in\R$ is a fixed real number. 

\item (\textit{A class of unbounded functions}).  Suppose that   $f\in L^1_{\textrm{loc}}(\R)$ and there exists  $K>0$ and $q\leq 1$ such that   $f\in C^3([-K,K]^c)$ and $|f'(x)|, |f''(x)|, |f'''(x)|\leq C$ and $|f'(x)|\leq C |x|^{q-1}$   for $|x|>K$. Then   $f$ satisfies (B) with $p=q$ when $q> 0$, and $p=0$ when $q<0$. This covers, in particular, the power functions $f(x)=|x|^q\1_{\{x\neq 0\}}$ 
where $q\in (-1,0)\cup (0,1]$, that is, $f_1$ and $f_2$ from \eqref{functions}.  Furthermore, the logarithmic function $f_5(x)=\log (|x|)\1_{\{x\neq 0\}}$ from \eqref{functions} is also covered by the above condition and hence satisfies (B). In this case we may  choose any $p\in (0,1]$. \qed
\end{enumerate}
\end{remark}

In the following we will need  to strengthen  Assumption (A). \\ \\
\textbf{Assumption (A2):} {\it Suppose
that in addition to Assumption (A) we have $|g^{(k)}(t)|\leq Ct^{\alpha-k}$ for all $t>0$. For the function $\zeta:[0,\infty)\to \R$ defined as $\zeta(t)=g(t)t^{-\alpha}$ the limit $\lim_{t\downarrow 0} \zeta^{(j)}$ exists in $\R$ for all $j=0,...,k$.} \\ \\
In the following two theorems we present weak limit results associated with Theorem~\ref{th1}(ii) in the case of ``short memory'' (small $\alpha$) or ``long memory'' (large $\alpha$). The long memory case depends heavily on the Appell rank of the function $f$, whereas the short memory case does not depend on the Appell rank.  
In the theorems below we follow the notation of Theorem \ref{th1}, i.e. $L$ is a symmetric $\beta$-stable  L\'evy process with scale parameter $\rho_L$, $(X_t)$ is given by \eqref{def-of-X-43},  $H=\alpha+1/ \beta$,  $\rho_0=\rho_L\|h_k\|_{L^\beta(\R)}$,  $S\sim S\beta S(1)$, $a_n=1/n$
and $b_n=n^H$.

\begin{theo}[``Short memory'']\label{lsjdlfjssss}

Suppose  Assumption~(A2) holds,      Assumption~(B) holds with $p<\beta/2$, and $\E[f(L_1)^2]<\infty$.  Then for all   $\alpha\in(0,k-2/\beta)$ we have
\begin{equation}\label{clt}
\sqrt n\bigg( V(f;k)^n - \E[f(\rho_0 S)]\bigg)\tol \mathcal N(0,\eta^2),
\end{equation}
where the variance is given as $\eta^2:=\lim_{m\to\infty}\eta_m^2$ with $\eta_m$ defined in \eqref{eta_mDef}. 
\end{theo}

\begin{theo}[``Long memory'']\label{the2ndGenVar}
Suppose that Assumption~(A2)  is  satisfied.    

\noindent 
(i) (Appell rank=1). Assume that $m^{\star}_{\rho_0}= 1$  and   Assumption~(B) holds with $p=1$.   For $\beta\in (1,2)$ and  $\alpha\in(k-1,k-1/\beta)$  we have that  
\begin{equation}\label{lsjdfgioskdow}
n^{k-\alpha-1/\beta}\bigg( V(f;k)^n - \E[f(\rho_0 S)]\bigg)\tol S\beta S(\sigma),
\end{equation}
where the scale parameter $\sigma$ is given by \eqref{ljsdfhhjljeri}. 

\noindent
(ii) (Appell rank$>$1). Suppose that Assumption~(B) holds with $p<\beta/2$, and 
$0=\frac{\partial }{\partial x}\Phi_\rho(0)=\frac{\partial^2 }{\partial x\partial \rho} \Phi_\rho(0)$
for all  $\rho\in (0,\infty)$.  For all  $\alpha\in(k-2/\beta,k-1/\beta)$ it holds that
 \begin{equation}\label{slt}
n^{1-\frac{1}{(k-\alpha)\beta}}\bigg(V(f;k)^n - \E[f(\rho_0 S)]\bigg)\tol \mathcal S((k-\alpha)\beta, 0, \rho_1,\eta_1),\qquad 
\end{equation}
 where the right hand side denotes the $(k-\alpha)\beta$-stable distribution with location parameter 0, scale parameter $\rho_1$ and skewness parameter $\eta_1$, which are specified in \eqref{ScaSkeParslt}. 
\end{theo}

\begin{remark} \rm \label{cltremarks} 

\begin{enumerate}[wide, labelwidth=!, labelindent=0pt, label=(\roman*)]
\item We note that the limiting distribution in Theorem~2.6(i)  is only non-degenerate in the Appell rank one case, or more precisely when $\frac{\partial }{\partial x}\Phi_{\rho_0}(0)\neq 0$, which follows from  \eqref{ljsdfhhjljeri}.
\item We also remark that 
the condition $m^{\star}_\rho \geq 2$  in Theorem~2.6(ii) is required to hold for \textit{all} $\rho>0$, which is in strong contrast
to the discrete framework of e.g.\ \cite{Sur2004} where only assumptions on  $m^{\star}_{\rho_0}$ are made.  The reason for our stronger condition on the Appell rank is the fact that the scaled increments $n^H \Delta_{i,k}^n X$ are only asymptotically S$\beta$S($\rho_0$)-distributed. 
\item Theorems \ref{lsjdlfjssss} and \ref{the2ndGenVar}(ii) give a rather complete picture of possible limits when the Appell rank is strictly large than one. Indeed, we cover all cases $\alpha \in (0, k-1/\beta)$ except the critical value of $\alpha=k-2/\beta$. This is not the case for the setting of Appell rank one. Not only we need to assume that $\beta \in (1,2)$, but we also have that 
$k-2/\beta < k-1$. Hence, the limit theory in the framework of $\beta \in (0,1]$, and also $\beta \in (1,2)$ with $\alpha \in [k-2/\beta, k-1]$, is still an open problem. 
\item Notice that Theorem \ref{lsjdlfjssss}, which has the fastest rate of convergence, never holds for $k=1$ since $\beta \in (0,2)$. 
Hence, for the purpose of statistical estimation, it makes sense to use higher values of $k$ to end up in the setting of Theorem \ref{lsjdlfjssss}. We refer  to \cite{MazOtrPod2017} for more details on statistical applications using higher order increments. \qed
\end{enumerate}
\end{remark}

 Similarly to Proposition \ref{prop1} one might be able to prove the functional versions of 
 Theorems~\ref{lsjdlfjssss} and \ref{the2ndGenVar}. 
However, we dispense with the precise exposition of these results in this paper.

\subsection{Outline of the proofs of Theorems~\ref{lsjdlfjssss} and \ref{the2ndGenVar}}

The strategy of the three proofs Theorems~\ref{lsjdlfjssss}, \ref{the2ndGenVar}(i) and \ref{the2ndGenVar}(ii) are  quite different, and are briefly outlined in the following. 

\begin{itemize}
\item 
For the proof of Theorem~\ref{lsjdlfjssss} we approximate $V(f;k)^n$ by 
\begin{align}\label{slfjsldhfdgd-sdfsd}
  V_{n,m}={}& \sum_{r=k}^n \big(f(n^{H} \din  X^m)-\E[f(n^{H} \din  X^m)]\big),\quad\text{where}\\ \quad X^{m}_t={}& 
   \int_{t-m/n}^t \big\{g(t-s)-g_0(-s)\big\}\, dL_s,\end{align}
More precisely, the main part of the proof is to show
\[\lim_{m\to\infty} \limsup_{n\to\infty}\E[n^{-1}(V(f;k)^n-V_{n,m})^2] =0.\]
It is then sufficient to establish asymptotic normality of $(V_{n,m})_{n\in\N}$ for each $m\geq 1$, which follows by the central limit theorem for $m$-dependent sequences of random variables. This general approach to deriving central limit theorems is popular in the literature, see \cite{PipTaq2003} for an example.
\item The main idea of the proof of Theorem~\ref{the2ndGenVar}(i) is to approximate $V(f;k)^n$, in a suitable sense,  by  a  linear functional $V_n$ of $(n^{H} \din X)_{i=k}^n$ given by 
\begin{equation}
V_n  = c_n \sum_{i=k}^n n^{H} \din X,\qquad n\in \N,
\end{equation}
where $c_n$ are certain chosen constants. 
With such an approximation in hand,  the proof boils down to showing that the $S\beta S$-stable random variables 
$V_n$ converge in distribution.

\item For the proof of Theorem~\ref{the2ndGenVar}(ii)  we decompose $V(f;k)^n$ as 
\begin{equation}\label{dslfjskgefe}
V(f;k)^n = \sum_{r=k}^n K_r+\sum_{r=k}^n Z_r
\end{equation}
where $\{Z_r\}_{k\geq n}$ is suitable defined i.i.d.\ sequence of random variables to be defined in \eqref{ZDef} below. 
We argue that the first  sum, on the right-hand side of \eqref{dslfjskgefe},  is asymptotically negligible and that the random variables $Z_r$ are in the domain of attraction of a $(k-\alpha)\beta$-stable random variable with location parameter 0, scale parameter $\rho_1$ and skewness parameter $\eta_1$ as defined in \eqref{ScaSkeParslt} in the proof.
Similar decompositions have been applied to derive stable limit theorems for discrete time moving averages, see for example \cite{HoHsi1997}.
\end{itemize}

\section{Preliminaries}
\label{sec2a}
\setcounter{equation}{0}
\renewcommand{\theequation}{\thesection.\arabic{equation}}

Throughout all our proofs we denote by $C$ a generic positive constant that does not depend on $n$ or $\omega$, but may change from line to line. 
For a random variable $Y$ and $q>0$ we denote $\|Y\|_q=\E[|Y|^q]^{1/q}.$ Throughout this paper we will repeatedly use the fact that if   $L$ is a symmetric $\beta$-stable L\'evy process with scale parameter $\rho_L$, then for each  function $\psi\in L^\beta(ds)$  the integral $\int_\R \psi(s)\,dL_s$ is a symmetric $\beta$-stable random variable  with scale parameter
\begin{equation}\label{ScaParInt}
\rho_L\bigg(\int_\R |\psi(s)|^\beta \,ds\bigg)^{1/\beta}=\rho_L\|\psi\|_{L^\beta(\R)}, 
\end{equation}
see \cite[Proposition 3.4.1]{SamTaq1994}. 
We will also frequently use the notation
\begin{align} \label{gindef}
g_{i,n}(s):= \sum_{j=0}^k (-1)^{j}\binom{k}{j}g((i-j)/n-s),
\end{align}
which leads to the expression
\[\din X=\int_{-\infty}^{i/n} g_{i,n}(s)\diff L_s\]
for the the $k$th order increments of $X$.
For the functions $g_{i,n}$ we obtain the following  estimates.
\begin{lem} \label{g_{i,n}Est}
Suppose that Assumption (A) is satisfied. It holds that
\begin{align*}
|g_{i,n}(s)|&\leq C(i/n-s)^\alpha\hspace{7.25em}\text{for }s\in[(i-k-1)/n,i/n],\\
|g_{i,n}(s)|&\leq Cn^{-k}((i-k)/n-s)^{\alpha-k}\qquad\text{for }s\in(i/n-\delta,(i-k-1)/n),\text{ and}\\
|g_{i,n}(s)|&\leq Cn^{-k}\big(\mathds 1_{[(i-k)/n-\delta,i/n-\delta]}(s)+g^{(k)}((i-k)/n-s)\mathds 1_{(-\infty,(i-k)/n-\delta)}(s)\big),\\
&\hspace{13.5em}\text{for }s\in(-\infty, i/n-\delta].
\end{align*}
\end{lem}
\begin{proof}
The  first inequality  follows directly from Assumption (A). The second inequality
is a straightforward consequence of Taylor expansion of order $k$ and the condition
$|g^{(k)}(t)|\leq C t^{\alpha-k}$ for $t\in (0,\delta)$. The third inequality  follows again through Taylor expansion and the fact that the function $g^{(k)}$ is decreasing on $(\delta, \infty)$.
\end{proof}
We  briefly recall the definition and some properties of the Skorokhod $M_1$-topology, as it is not as widely used as the $J_1$-topology.  
It was originally introduced by Skorokhod \cite{Sko1956} by defining a metric on the completed graphs of c\`adl\`ag functions, where the completed graph of $\phi$ is defined as 
\[\Gamma_\phi =\{(x,t)\in \R\times\R_+\,:\, x=\alpha \phi(t-)+(1-\alpha)\phi(t),\text{ for some }\alpha\in[0,1]\}.\]
The $M_1$-topology is weaker than the  $J_1$-topology but still strong enough to make many important functionals, such as supremum and infimum, continuous. It can be shown that the stable convergence in Theorem \ref{th1}(i) does not hold with respect to the $J_1$-topology (cf.\ \cite{BasHeiPod2017}). Since the $M_1$-topology is metrizable, it is completely characterized through  convergence  of sequences, which we describe  in the following.
A sequence $\phi_n$ of functions in $\mathbb D(\R_+,\R)$ converges to $\phi\in \mathbb D(\R_+,\R)$ with respect to the Skorokhod $M_1$-topology if and only if $\phi_n(t)\to \phi(t)$ for all $t$ in a dense subset of $[0,\infty)$, and for all $t_\infty\in[0,\infty)$ it holds that
\[\lim_{\delta_\downarrow 0}\limsup_{n\to\infty} \sup_{0\leq t\leq t_\infty} w(\phi_n,t,\delta) =0.\]
Here, the oscillation function $w$ is defined as
\begin{align}\label{oscfun}
w(\phi,t,\delta)=\sup_{0\vee(t-\delta)\leq t_1<t_2<t_3\leq (t+\delta)\wedge t_\infty}\{|\phi(t_2)-[\phi(t_1),\phi(t_3)]|\},
\end{align}
where for $b<a$ the interval $[a,b]$ is defined to be $[b,a], $ and $|a-[b,c]|:=\inf_{d\in[b,c]}|a-d|.$ We refer to \cite{Whi2002}
for more details on the $M_1$-topology.


\section{Proof of Theorem \ref{th1} }
\label{secPro1}
\setcounter{equation}{0}
\renewcommand{\theequation}{\thesection.\arabic{equation}}



\subsection{Proofs of Theorem~\ref{th1}(i) and Proposition~\ref{prop1}(i)}


We concentrate on the proof of Proposition \ref{prop1}(i), since it is a stronger statement than 
Theorem \ref{th1}(i). 
The proof is divided into three parts. First, we assume that $L$ is a compound Poisson process and show the finite dimensional stable convergence for the statistic $V(f;k)_t^n$.
Thereafter we argue that the convergence holds in the functional sense with respect to the $M_1$-topology, when $f$ satisfies condition \ref{con3}.
 Finally, the results are extended to general L\'evy processes by truncation. For this step, an isometry for L\'evy integrals, which is due to \cite{RajRos1989}, plays a key role.

Since $C^q(\R)\subset  C^p(\R)$ for $p<q$ we may and do assume that $p\not\in\N$. Note that, if  
$f\in C^p(\R)$ and $f^{(j)}(0)=0$ for all $j=0,\dots, [p]$, then for any $N>0$ there exists a constant $C_N$ such that
\begin{align}\label{fEstStaCon}
|f^{(j)}(x)|\leq C_N|x|^{p-j},\quad \text{ for all }x\in[-N,N], \text{ and }j=0,\dots [p].
\end{align}
By the assumption $p>\frac 1 {k-\alpha}$, this implies the following estimate to be used in the proof below.
For all $N>0$ there is a constant $C_N$ such that
\begin{align}\label{fEstStaCon2}
|f^{(j)}(x)|\leq C_N|x|^{\gamma_j},\quad \text{ for all }x\in[-N,N], \text{ and }j=0,\dots ,[p],
\end{align}
where $\gamma_j= \frac{p-j}{p(k-\alpha)}.$
The following lemma ensures in particular that the limit in Theorem~\ref{th1}(i) exists.
\begin{lem}\label{LimEx} Let $t>0$ be fixed. Under conditions of Theorem \ref{th1}(i) there exists a finite random variable $K>0$ such that
\begin{align}\label{lsjdfljsdhsdlh}
\sum_{m:\,T_m\in [0,t]} \sum_{l=0}^\infty \big|f\big( \Delta L_{T_m} h_k(l+U_m)\big)\big|&\leq K,\quad\text{and}\\
\label{lsjdsfgsdfghsdfsde}\sum_{m:\,T_m\in [0,t]} \sum_{l=0}^{n-1} \big|f\big( \Delta L_{T_m} n^\alpha g_{i_m+l,n}(T_m)\big)\big|&\leq K,\quad\text{for all $n$,}
\end{align}
where $i_m$ denotes the random index such that $T_m\in\big(\frac{i_m-1}n,\frac{i_m}n\big].$
\end{lem}
\begin{proof}
Throughout the proof, $K$ denotes a positive random variable that does not depend on $n$, but may change from line to line.   For the first inequality note that $|h_k(l+U_m)|\leq C(l-k)^{\alpha-k}$ for all $l>k$ and $|h_k(l+U_m)|\leq C$ for $l\in\{0,...,k\}$. This implies in particular 
\begin{align*}
| \Delta L_{T_m}(\omega) h_k(l+U_m)|\leq 
\begin{cases} C(l-k)^{\alpha-k}\sup_{s\in[0,t]}|\Delta L_s| ,&\text{for }l>k\\
 C\sup_{s\in[0,t]}|\Delta L_s|,&\text{for }l\in\{0,...,k\}.
 \end{cases}
\end{align*}
 Therefore, we find by \eqref{fEstStaCon} a random variable $K$ such that 
\[\big|f\big( \Delta L_{T_m} h_k(l+U_m)\big)\big|\leq K\big| \Delta L_{T_m} h_k(l+U_m)\big|^p\]
for all $l\geq 0$ and all $m.$
Consequently, the left-hand side of \eqref{lsjdfljsdhsdlh}  is dominated by
\[K\bigg(\sum_{m:\,T_m\in [0,t]}|\Delta L_{T_m}|^p+\sum_{m:\,T_m\in [0,t]}|\Delta L_{T_m}|^p\sum_{l=k+1}^\infty(l-k)^{(\alpha-k)p}\bigg)\leq K,\]
where we used that $(\alpha-k)p<-1,$ and that $\sum_{m:\, T_m \in [0,t]} |\Delta L_{T_m}|^p<\infty$ since $p>\beta$. The  inequality \eqref{lsjdsfgsdfghsdfsde} follows by the same arguments since Lemma \ref{g_{i,n}Est} implies the existence of a constant $C>0$ such that for all $n\in\N$
\begin{align*}
&n^{\alpha }g_{i_m+l,n}(T_m)\leq C &&\text{ for }l\in\{0,...,k\},\text{ and }\\
&n^{\alpha }g_{i_m+l,n}(T_m)\leq C(l-k)^{\alpha-k},&&\text{ for }l\in\{k+1,...,n-1\} .
\end{align*}
\end{proof}


\subsubsection{Compound Poisson process as driving process}


In this subsection, we show the finite dimensional stable convergence of $V(f;k)_{t}^n$  under the assumption that $L$ is a compound Poisson process. The extension to functional convergence when condition \ref{con3} is satisfied
follows in the next subsection, the extension to general $L$ thereafter.

 Let $0\leq T_1<T_2<...$ denote the jump times of $(L_t)_{t\geq 0}$.
For $\eps>0$ we define
\begin{align*}
 \Omega_\eps=\big\{\omega \in\Omega : &\text{ for all $m$ with $T_m(\omega)\in[0,t]$ we have $|T_m(\omega)-T_{m-1}(\omega)|>\eps$}\\
&\text{ and $\Delta L_s(\omega)=0$ for all $s\in [-\eps,0]$ and $|\Delta L_s(\omega)|\leq \eps^{-1}$ for all $s\in[0,t]$}\big\}.
\end{align*}
We note that $\Omega_\eps\uparrow \Omega,$ as $\eps\downarrow 0.$
Letting
\begin{align}\label{MRdef}
M_{i,n,\eps}:=\int_{i/n-\eps}^{i/n} g_{i,n}(s)\diff L_s,\quad\text{and}\quad R_{i,n,\eps}:=\int^{i/n-\eps}_\infty g_{i,n}(s)\diff L_s,
\end{align}
we have the decomposition $\din X=M_{i,n,\eps}+R_{i,n,\eps}.$ It turns out that $M_{i,n,\eps}$ is the asymptotically dominating term, whereas $R_{i,n,\eps}$ is negligible as $n\to\infty.$
We show that, on $\Omega_\eps$,
\begin{align}\label{MLim}
\sum_{i=k}^{[nt]} f(n^{\alpha }M_{i,n,\eps})\fidi Z_{t},\quad\text{where}\quad Z_{t}:=\sum_{m:\,T_m\in [0,t]}\sum_{l=0}^\infty f(\Delta L_{T_m}h_k(l+U_m) ),
\end{align}
as $n \to \infty$. 
Here $(U_m)_{m\geq 1}$ are independent identically $\mathcal U([0,1])$-distributed random variables, defined on an extension $(\Omega',\mathcal F',\P')$ of the original probability space, that are independent of $\mathcal F$. 
For this step, the following expression for the left hand side is instrumental.
On $\Omega_\eps$ it holds that
\begin{align}\label{sumM=V}
\sum_{i=k}^{[nt]} f(n^\alpha M_{i,n,\eps})=V^{n,\eps}_t,
\end{align}
where
\begin{align}\label{V^nDef}
V^{n,\eps}_t:=  \sum_{m:T_m\in (0,[nt]/n]}\sum_{l=0}^{v_t^m}f(n^{\alpha}\Delta L_{T_m}g_{i_m +l,n}(T_m)).
\end{align}
Here, $i_m$ denotes the random index such that $T_m\in((i_m-1)/n,i_m/n]$, and $v^m_{t}$ is defined as
\begin{align}\label{v_t^mdef}
 v^m_{t }= v_t^m(\eps,n):=
\begin{cases} 
[\eps n]\wedge ([nt]-i_m) & \text{if } T_m-([\eps n]+i_m)/n>-\eps, \\
[\eps n]-1 \wedge ([nt]-i_m) & \text{if }T_m-([\eps n]+i_m)/n\leq-\eps.
\end{cases}
\end{align}
Additionally, we set $v^m_{t }=\infty$ if $T_m>[nt]/n.$ The following lemma proves \eqref{MLim}.\begin{lem}\label{fidi}
For $r\geq 1$ and $0\leq t_1<\dots< t_r\leq t$ we obtain on $\Omega_\eps$ the stable convergence
\[(V^{n,\eps}_{t_1},\dots, V^{n,\eps}_{t_r})\tols (Z_{t_1},\dots, Z_{t_r}),\quad \text{ as }n\to\infty.\]
\end{lem}
\begin{proof}
 By arguing as in \cite[Section 5.1]{BasLacPod2016}, we deduce for any $d\geq 1$ the $\mathcal F$-stable convergence
\[\{n^\alpha g_{i_m+l,n}(T_m)\}_{l,m\leq d}\tols \{ h_k(l+U_m)\}_{l,m\leq d}\]
as $n\to\infty$.
Defining
\begin{align*}
V_t^{n,d}&:= \sum_{m\leq d:\,T_m\in (0,[nt]/n]}\sum_{l=0}^{d} f(n^\alpha \Delta L_{T_m}g_{i_m +l,n}(T_m))\quad\text{and}\\
Z^d_t& := \sum_{m\leq d:\,T_m\in (0,t]}\sum_{l=0}^{d} f(\Delta L_{T_m}h_k(l+U_m)),
\end{align*}
we obtain by the continuous mapping theorem for stable convergence 
\begin{align}\label{fidiconsta}
(V^{n,d}_{t_1},\dots, V^{n,d}_{t_r})\tols (Z^d_{t_1},\dots, Z^d_{t_r}),\quad \text{ as }n\to\infty,
\end{align}
for all $d\geq 1.$ Therefore, by a standard approximation argument (cf.\ \cite[Theorem~3.2]{Bil1999}), it is sufficient to show that
\begin{align} 
\label{V^nvsV^nd}
&\limsup_{n\to\infty} \bigg\{\max_{t\in\{t_1,\dots,t_r\}}|V_t^{n,\eps}-V_t^{n,d}|\bigg\}\toas 0,&&\text{ as }d\to\infty,\quad\text{and}
\\ \label{Zcon1}
&\sup_{s\in[0,t]}|Z^d_s-Z_s|\toas 0,&&\text{ as }d\to\infty.
\end{align}
For all $s\in[0,t]$ and sufficiently large $n$ we have
\begin{align}
|V^{n,d}_s-V^{n,\eps}_s|
\leq\ &\sum_{m\leq d: \,T_m\in(0,[ns]/n]}\sum_{l=d\wedge v_{t}^m}^{d\vee v_t^m}|f(\Delta L_{T_m}n^{\alpha} g_{i_m+l,n}(T_m))|\nonumber\\
 &+\sum_{m>d:\,T_m\in(0,[ns]/n]}\sum_{l=0}^{ v_t^m}|f(\Delta L_{T_m}n^{\alpha} g_{i_m+l,n}(T_m))|\nonumber\\
\leq\ & \sum_{m: \, T_m\in(0,t]}\sum_{l=d\wedge v_{t}^m}^{n-1}|f(\Delta L_{T_m}n^{\alpha} g_{i_m+l,n}(T_m))|\nonumber\\
 &+\sum_{m>d:\,T_m\in(0,[nt]/n]}\sum_{l=0}^{ n-1}|f(\Delta L_{T_m}n^{\alpha} g_{i_m+l,n}(T_m))|.\nonumber
\end{align}
Therefore, \eqref{V^nvsV^nd} follows from Lemma \ref{LimEx} by the dominated convergence theorem 
since the random index $v_t^m=v_t^m(n,\omega)$ satisfies $\liminf_{n\to\infty} v^m_t(n,\omega)=\infty$, almost surely. 
Lemma \ref{LimEx} also implies \eqref{Zcon1}, since
\begin{align*}
&\sup_{s\in[0,t]}|Z^d_s-Z_s|\leq\\
 &\sum_{m\leq d:\,T_m\in (0,t]}\sum_{l=d+1}^{\infty} |f(\Delta L_{T_m}h_k(l+U_m))|+ \sum_{m> d:\,T_m\in (0,t]}\sum_{l=0}^{\infty} |f(\Delta L_{T_m}h_k(l+U_m))|.
\end{align*} 
The lemma now follows from \eqref{fidiconsta}, \eqref{V^nvsV^nd} and \eqref{Zcon1}.
\end{proof}

Recalling the decomposition \eqref{MLim} and applying the triangle inequality, the proof can be completed by showing that 
\begin{align}\label{VvsM}
J_n:=\sum_{i=k}^{[nt]} |f(n^\alpha \din X)-f(n^\alpha M_{i,n,\eps})| \toas 0,
\end{align}
as $n\to \infty$. We first argue that the random variables $\{n^\alpha M_{i,n,\eps},n^\alpha \din X\}_{n\in\N,i\in\{k,...,[nt]\}}$ are  uniformly bounded by a constant on $\Omega_\eps$, which will allow us to apply the estimate \eqref{fEstStaCon}. The random variables $M_{i,n,\eps}$ satisfy by construction either $|n^\alpha M_{i,n,\eps}|=0$ or $|n^\alpha M_{i,n,\eps}|=|n^\alpha g_{i,n}(T_m)\Delta L_{T_m}|$ for some $m$, where we recall that on $\Omega_\eps$ it holds that $T_m-T_{m-1}>\eps$. Consequently, they are uniformly bounded by Lemma \ref{g_{i,n}Est}, where we used that $k>\alpha$ and that the jumps of $L$ are bounded on $\Omega_\eps$. The uniform boundedness of $n^\alpha\din X=n^\alpha(M_{i,n,\eps}+R_{i,n,\eps})$ follows by \cite[Eqs. (4.8), (4.12)]{BasLacPod2016} which implies that for any $\eta>0$ 
\begin{align}\label{Rbound}
\sup_{n\in\N,\ i\in\{k,\dots,[nt]\}}\big\{n^{k-\eta}|R_{i,n,\eps}|\big\}<\infty,\quad \text{almost surely.}
\end{align}

In order to show \eqref{VvsM} we apply Taylor expansion for $f$ at $n^\alpha M_{i,n,\eps}$, and bound the terms in the Taylor expansion using \eqref{fEstStaCon} and the following lemma.
\begin{lem}\label{TayEst}
Let $\psi:\R\to\R$ be continuous and such that $|\psi(x)|\leq C|x|^{\gamma}$ for all $x\in[-1,1]$ for some $\gamma\in (0,1/(k-\alpha))$. It holds on $\Omega_\eps$ that
\[\limsup_{n\to\infty}\bigg\{ n^{(k-\alpha)\gamma-1}\sum_{i=k}^{[nt]} |\psi(n^\alpha M_{i,n,\eps})|\bigg\}\leq C,\quad \text{a.s.} \]
\end{lem}
\begin{proof}
We have on $\Omega_\eps$ 
\[\sum_{i=k}^{[nt]} |\psi(n^\alpha M_{i,n,\eps})|=W^{n,\eps}_{t},\]
where
\[W^{n,\eps}_{t}:=  \sum_{m:T_m\in (0,[nt]/n]}\sum_{l=0}^{v_{t_\infty}^m}|\psi(n^{\alpha}\Delta L_{T_m}g_{i_m +l,n}(T_m))|,\]
and $v_{t_\infty}^m$ is the random index defined in \eqref{v_t^mdef}. By Lemma \ref{g_{i,n}Est} the random variables $n^{\alpha }g_{i_m+l,n}(T_m)$ are bounded for $l=0,...,k$. For $l\in\{k+1,...,n-1\}$, Lemma \ref{g_{i,n}Est} implies that $n^{\alpha }g_{i_m+l,n}(T_m)\leq C(l-k)^{\alpha-k}$. Since the random index $v_{t_\infty}^m$ satisfies $v_{t_\infty}^m< n$ for all $m$, we obtain on $\Omega_\eps$
\[\sum_{i=k}^{[nt]} |\psi(n^\alpha M_{i,n,\eps})|\leq C\sum_{m:T_m\in (0,t]}\bigg(\sum_{l=0}^{k}|n^\alpha g_{i_m+l,n}(T_m)|^\gamma +\sum_{l=k+1}^{n}|(l-k)^{\alpha-k}|^\gamma\bigg).\]
It follows by comparison with the integral $\int_{k+1}^n(s-k)^{(\alpha-k)\gamma}\diff s$ that the right hand side multiplied with $n^{(k-\alpha)\gamma-1}$ is convergent, where we used that $(\alpha-k)\gamma\in (-1, 0)$ and that the number of jumps of $L(\omega)$ in $[0,t]$ is uniformly bounded for 
$\omega\in\Omega_\eps$.
\end{proof}

Considering the sum $J_n$ in \eqref{VvsM}, Taylor expansion up to order $r=[p]$ shows that
\begin{align}\label{Taylor}
J_{n}
&\leq \sum_{i=k}^{[nt]} \big| n^\alpha R_{i,n,\eps} f'(n^{\alpha} M_{i,n,\eps})\big| +\dots+ \frac 1{r!}\sum_{i=k}^{[nt]} \big| (n^\alpha R_{i,n,\eps})^r f^{(r)}(n^{\alpha} M_{i,n,\eps})\big|+TR_r\nonumber\\
&:=T_1+\dots +T_r+ TR_r,
\end{align}
where $TR_r$ denotes the Taylor rest term.
Recalling the estimate \eqref{fEstStaCon2}, we can now estimate the $j$th Taylor monomial $T_j$
for $j=0,\dots,[p]$ by applying Lemma~\ref{TayEst} on $\psi=f^{(j)}$, where we remark that $\gamma_j=\frac{p-j}{p(k-\alpha)}\in (0,1/(k-\alpha))$. 
Using \eqref{Rbound} and recalling that $p>k-\alpha$, we obtain that for sufficiently small $\eta>0$
\begin{align}\label{TayMon}
 \frac 1 {j!}\sum_{i=k}^{[nt]} \big| (n^\alpha R_{i,n,\eps})^j f^{(j)}(n^{\alpha} M_{i,n,\eps})\big|
	&\leq C n^{-j/p-\eta}\sum_{i=k}^{[nt]}| f^{(j)}(n^{\alpha} M_{i,n,\eps})\big|\nonumber\\
	&\leq C n^{-\eta},
\end{align}
where the second inequality follows from Lemma~\ref{TayEst} since $(k-\alpha)\gamma_j-1 = -j/p.$
 For the Taylor rest term $TR_r$ we obtain by the mean value theorem:
\[TR_r = \frac 1 {r!}\sum_{i=k}^{[nt]} \big| (n^\alpha R_{i,n,\eps})^{r}  \big(f^{(r)}(\xi_{i,n})-f^{(r)}(n^\alpha M_{i,n,\eps})\big)\big|,\]
with $\xi_{i,n}\in (n^\alpha |M_{i,n,\eps}|,n^\alpha |X_{i,n,\eps}|)$ where we set $(a,b):=(b,a)$ for $a>b$. Since $n^\alpha |M_{i,n,\eps}|$ and $n^\alpha|X_{i,n,\eps}|$ are bounded and $f^{(r)}$ is locally 
$(p-r)$-H\"older continuous, it follows that
\[TR_r\leq Cn\sup_{n\in\N,\ i\in\{k,\dots,[nt]\}} |n^\alpha R_{i,n,\eps}|^p.\]
From \eqref{Rbound} it follows that $TR_r\to 0$ as $n\to\infty$, where we recall that $(\alpha-k)p<-1.$ 
Together with \eqref{Taylor} and \eqref{TayMon} this implies $J_n\toas 0$, and it follows that
\[\sup_{s\in[0,t]}\bigg\{\bigg|V(f;k)_s^n-\sum_{i=k}^{[ns]}f(n^\alpha M_{i,n,\eps})\bigg|\bigg\}\toas 0\]
on $\Omega_\eps.$ Now, the proposition follows from Lemma~\ref{fidi} by letting $\eps\to 0.$


\subsubsection {Functional convergence}


In this subsection we show that if $f$ satisfies \ref{con3} and under the assumption that $L$ is a compound Poisson process, the convergence in Proposition \ref{prop1}(i) holds in the functional sense with respect to the Skorokhod  $M_1$-topology. 
To this end, we denote by $\skorcon$ the stable convergence of c\`adl\`ag processes on $\mathbb D([0,t];\R)$ equipped with the Skorokhod $M_1$-topology.
 We first replace \ref{con3} by the following stronger auxiliary assumption.{\it
 \begin{enumerate}[label=(FC')]\setcounter{enumi}{2}
\item \label{con3'} 
It holds that $f$ is either non-negative or non-positive.
\end{enumerate}}
 This assumption puts us into the comfortable situation that our limiting process is monotonic. 
  Recall the definition of the processes $V^{n,\eps}$ and $ Z$ introduced in \eqref{MLim} and \eqref{V^nDef}, respectively.
In Lemma \ref{fidi} the stable convergence of the finite dimensional distributions of $V^{n,\eps}$ to $Z$ was shown. By Prokhorov's theorem the functional convergence $V^{n,\eps}\skorcon Z$ on $\Omega_\eps$ follows  from the following lemma.
 \begin{lem}\label{tight}
The sequence of $\D([0,t])$-valued random variables $(V^{n,\eps}\mathds 1_{\Omega_\eps})_{n\geq 1}$ is tight with respect to Skorokhod $M_1$-topology. 
\end{lem}
\begin{proof}
It is sufficient to show that the conditions of \cite[Theorem 12.12.3]{Whi2002} are satisfied.
Condition (i) is satisfied, since the family of real valued random variables $(V_{t}^{n,\eps})_{n\geq 1}$ is tight by Lemma \ref{fidi}. Condition (ii) is satisfied, since the oscillating function $w_s$ introduced in \cite[Chapter 12, (5.1)]{Whi2002} satisfies $w_s(V^{n,\eps},\theta)=0$ for all $\theta>0$ and all $n$, since $V^{n,\eps}$ is monotonic by assumption \ref{con3'}.
\end{proof}
Recalling the identity \eqref{sumM=V} and the asymptotic equivalence of $\sum_{i=k}^{[nt]}f(n^\alpha M_{i,n,\eps})$ and $V(f;k)^n_t$ shown in \eqref{VvsM} and thereafter, the functional convergence in Proposition \ref{prop1}(i) follows.

Now, for general $f$ satisfying condition \ref{con3} we decompose $f=f_++f_-$ with $f_+(x)=f(x)\mathds 1_{\{x>0\}}$ and $f_-(x)=f(x)\mathds 1_{\{x<0\}}.$ Both functions $f_+$ and $f_-$ satisfy \ref{con3'}, and the functional convergence of $V(f_+;k)^n$ and $V(f_-;k)^n$ follows, with the corresponding limits denoted by $Z^+$ and $Z^-$. Note that $Z^+$ jumps exactly at those times, where the L\'evy process $L$ jumps up, and $Z^-$ at those, where it jumps down. In particular, $Z^+$ and $Z^-$ do not jump at the same time, which implies that summation is continuous at $(Z^+,Z^-)$ with respect to the $M_1$-topology (cf.\ \cite[Theorem~12.7.3]{Whi2002}). Thus, an application of the continuous mapping theorem yields the convergence of $V(f;k)^n=V(f_+;k)^n+V(f_-;k)^n$ towards $Z=Z^++Z^-.$ Let us stress that indeed the sole reason why the extra condition \ref{con3} is required for functional convergence is that summation is not continuous on the Skorokhod space in general, and the convergence of $V(f_+;k)^n$ and $V(f_-;k)^n$ does not necessarily imply the convergence of $V(f;k)^n$.


\subsubsection{Extension to infinite activity L\'evy processes}


In this section we extend the results of Proposition~\ref{prop1}(i) to moving averages driven by a general L\'evy process $L,$ by approximating $L$ by a sequence of compound Poisson processes $(\hat L(j))_{j\geq 1}.$ To this end we introduce the following notation. Let $N$ be the jump measure of $L$, that is $N(A):=\#\{t:(t,\Delta L_t)\in A\}$ for measurable $A\subset \R\times (\R\setminus \{0\}),$ and define for $j\in \N$
\[X_t(j):= \int_{(-\infty,t]\times [-\frac 1 j, \frac 1 j]}\{(g(t-s)-g_0(-s))x\}N(ds,dx).\]
Denote $\hat X_t(j):= X_t-X_t(j).$ The results of the last section show that Proposition \ref{prop1}(i) holds for $\hat X(j),$ since it is a moving average driven by a compound Poisson process. By letting $j\to\infty$ we will show that the theorem remains valid for $X$ by deriving the following approximation result
\begin{lem}\label{lemSmaJumApp} Suppose that $f$ satisfies the conditions of 
Proposition \ref{prop1}(i). It holds that
\begin{equation}\label{SmaJumApprox}
\lim_{j\to\infty} \limsup_{n\to\infty} \P\bigg(\sup_{s\in[0,t]}|V(X,f;k)^n_s-V(\hat X(j),f;k)^n_s|>\eps\bigg)=0,\quad \text{for all }\eps>0.
\end{equation}      
\end{lem}
\begin{proof}
In the following we say that a family $\{Y_{n,j}\}_{n,j\in\N}$ of random variables is {\it asymptotically tight} if for any $\eps>0$ there is an $N>0$ such that
\[\limsup_{n\to\infty}\P(|Y_{n,j}|>N)<\eps,\quad \text{for all $j\in\N$}.\]
We deduce first for $p>\beta\vee \frac{1}{k-\alpha}$ the asymptotic tightness of the two families
 \begin{align}
&\bigg\{ \sum_{i=k}^{[nt]}|n^{\alpha}\din \hat X(j)|^{p}\bigg\}_{n,j\in\N}\quad\text{and}\quad\bigg\{\max_{i= k,...,[nt]}|n^\alpha \din \hat X(j)|\bigg\}_{n,j\in\N}, \label{AsyTig1}
\intertext{and tightness of}
\label{sdlfjsdhhs}
 &\bigg\{\sum_{i=k}^{[nt]}|n^{\alpha}\din  X|^{p}\bigg\}_{n\in \N}\quad \text{and}\quad \bigg\{\max_{i= k,...,[nt]} |n^\alpha \din X|\bigg\}_{n\in\N}.\end{align}
The authors of \cite{BasLacPod2016} show the stable convergences in law
\begin{equation}\label{sdlfjsdlhf}
 \sum_{i=k}^{[nt]} |n^\alpha \din \hat X(j)|^p\tols Z_j,\qquad \text{and}\qquad\sum_{i=k}^{[nt]} |n^\alpha \din X|^p\tols Z,
\end{equation}
where $Z_j$ and $Z$ are defined as in \cite[Eq.~(4.34)]{BasLacPod2016}. The asymptotic tightness of the first family of random variables in \eqref{AsyTig1} follows thus from  the tightness of the family $\{Z_j\}_{j\in\N}$, see \cite[Eq.~(4.35)]{BasLacPod2016}.
The asymptotic tightness of the second family of random variables from \eqref{AsyTig1}  follows from the first by the estimate $\max_{i=1,...,n}|a_i|\leq \big(\sum_{i=1}^n |a_i|^p\big)^{1/p}$ for $a_1,...,a_n\in\R.$ The second statement of \eqref{sdlfjsdlhf} implies \eqref{sdlfjsdhhs} by similar arguments. 
The (asymptotic) tightness of the two families on the right-hands side of \eqref{AsyTig1} and \eqref{sdlfjsdhhs}  allows us, for the proof of \eqref{SmaJumApprox}, to assume that 
$|\din \hat X(j)|$ and $|\din X|$ are uniformly bounded by some $N>0$. 

Consider first the case $p<1.$ By local H\"older-continuity of $f$ of order $p$ we have that
\[\sup_{s\in[0,t]}|V(f,X;k)^n_s-V(f,\hat X(j);k)^n_s|\leq C_N \sum_{i=k}^{[nt]}|n^{\alpha}\din X(j)|^{p}, \]
and \eqref{SmaJumApprox} follows from \cite[Lemma 4.2]{BasLacPod2016}, where we used that $p>\beta\vee \frac 1 {(k-\alpha)}$. 
Let now $p>1$. We can find $\xi_{i,n,j}\in[n^\alpha\din \hat X(j),n^\alpha\din  X]$ such that 
$|f(n^\alpha\din \hat X(j))-f(n^\alpha\din  X)|= |n^\alpha\din  X (j) f'(\xi_{i,n,j})|$ and with $\gamma=\frac{p-1}p\big(\beta\vee \frac{1}{k-\alpha}\big)$ we obtain by \eqref{fEstStaCon} that 
\begin{align*}
& |f(n^\alpha\din \hat X(j))-f(n^\alpha\din  X)|\leq C |n^\alpha\din  X (j)| |\xi_{i,n,j}|^{p-1}\\
&\qquad \leq C |n^\alpha\din  X (j)| |\xi_{i,n,j}|^{\gamma}
	\leq C|n^\alpha\din X(j)|^{\gamma+1}+C |n^\alpha\din X(j)||n^\alpha\din X|^{\gamma},
\end{align*}
since  $\gamma<p-1$ by assumption.
Thus, in order to complete the proof of \eqref{SmaJumApprox}, it is sufficient to show that for all $\eps>0$ we obtain 
\begin{align}
\label{SmaJumEst1}&\lim_{j\to\infty} \limsup_{n\to\infty} \P\bigg(\sum_{i=k}^{[nt]}|n^\alpha\din X(j)|^{\gamma+1}>\eps\bigg)=0,\quad\text{and}\\
\label{SmaJumEst2}&\lim_{j\to\infty} \limsup_{n\to\infty} \P\bigg(\sum_{i=k}^{[nt]}|n^\alpha\din X(j)||n^\alpha\din X|^{\gamma}>\eps\bigg)=0.
\end{align}
By definition it holds that $\gamma+1>\beta\vee \frac 1 {k-\alpha}$, and \eqref{SmaJumEst1} follows from \cite[Lemma 4.2]{BasLacPod2016}.
For \eqref{SmaJumEst2} we choose H\"older conjugates $\theta_1$ and $\theta_2=\theta_1/(\theta_1-1)$ with $\theta_1\in\big(\beta  \vee\frac 1 {k-\alpha},  p \big)$, where we used that $p>1$. 
The H\"older inequality and the estimate $\P(|XY|>\eps)\leq \P(|X|> \eps/N)+P(|Y|> N)$ for any $N>0$ lead to the decomposition
\begin{align*}
&\P\bigg(\sum_{i=k}^{[nt]}|n^\alpha\din X(j)||n^\alpha\din X|^{\gamma}>\eps\bigg)\\
& \eqspace\leq \P\bigg(\sum_{i=k}^{[nt]}|n^{\alpha}\din X(j)|^{\theta_1}>\bigg(\frac{\eps}{N}\bigg)^{\theta_1}\bigg) +
\P\bigg(\sum_{i=k}^{[nt]}|n^{\alpha}\din \hat X(j)|^{\gamma\theta_2}> N^{\theta_2}\bigg)\\
&\eqspace=: J^1_{n,j,N}+J^2_{n,j,N}.
\end{align*}
Since $\theta_1>\beta\vee \frac 1 {k-\alpha}$, yet another application of \cite[Lemma 4.2]{BasLacPod2016} yields that
\[\lim_{j\to\infty}\limsup_{n\to\infty} J^1_{n,j,N}=0\quad\text{ for all $N>0$, and all $\eps>0$.}\]
Moreover, $\theta_1<p$ implies $\gamma\theta_2>\beta\vee\frac 1{k-\alpha}$. Therefore, it follows from the asymptotic tightness of the first family of random variables from \eqref{AsyTig1} that
\[\limsup_{j\to\infty}\limsup_{n\to\infty}J^2_{n,j,N} \to 0,\quad\text{ as $N\to\infty,$}\]
which completes the proof of the lemma.
\end{proof}
 Finally, the proof of Proposition~\ref{prop1}(i) can be completed by letting $j\to\infty$. More precisely, we introduce for $j\in\N$ the stopping times
\[T_{m,j} := 
\begin{cases}
T_m & \text{if } |\Delta L_{T_m}|>1/j,\\
\infty & \text{else}.
\end{cases}
\]
The results of the last two subsections show that 
\[V(\hat X(j),f;k)^n_t\fidi Z_t^j:= \sum_{m: \,T_{m,j}\in[0,t]} \sum_{l=0}^\infty f(\Delta L_{T_{m,j}} h_k(l-U_{m})),\]
 and that the convergence holds in the functional sense with respect to the $M_1$-topology if $f$ satisfies \ref{con3}.
From Lemma \ref{LimEx} and an application of the dominated convergence theorem it follows that
\[\sup_{s\in[0,t]}|Z_s-Z_{s}^j|\toas 0,\quad\text{as } j\to\infty.\]
Proposition \ref{prop1}(i) follows therefore from Lemma~\ref{lemSmaJumApp} and a standard approximation argument
 (cf.\ \cite[Theorem~3.2]{Bil1999}).
\qed

\subsection{Proof of Theorem~\ref{th1}(ii)}\label{sldfjsdljfl}

As mentioned earlier the proof relies upon replacing the increments of $X$ by the increments of its tangent process, which is the linear fractional stable motion $Y$, defined as
\begin{align*} 
Y_t :=  \int_{\R} \{(t-s)_{+}^\alpha  - ( -s)_{+}^\alpha \}\, dL_s,
\end{align*} 
 It is well known that the process $Y$ is self-similar with index $H=\alpha+1/\beta$, i.e. $(Y_{at})_{t\geq 0}\eqdist(a^{H}Y_t)$ for any $a>0$, see \cite{TaqWol1983}. Moreover, the discrete time stationary sequence $(Y_r-Y_{r-1})_{r\in\Z}$ is mixing and hence ergodic, see for example \cite{CamHarWer1987}. Denoting by $V(f;Y)^n$ the variation functional \eqref{vn} with $a_n=n^{-1}$ and $b_n=n^H$ applied on the process $Y$, it follows from Birkhoff's ergodic theorem, see \cite[Theorem 10.6]{Kal2002}, that
\[V(f;Y)^n=\frac 1 n \sum_{i=k}^{n} f(n^H\din Y)\eqdist \frac 1 n \sum_{i=k}^{n} f(\Delta_{i,k}^1 Y)\toas \E[f(\Delta_{k,k}^1 Y)].\]
By \eqref{ScaParInt}, the random variable $\Delta_{k,k}^1 Y \sim S\beta S(\rho_0)$  with  
$\rho_0=\rho_L\|h_k\|_{L^{\beta}(\R)}$, and the right hand side is the limiting expression in 
Theorem \ref{th1}(ii).
It is therefore sufficient to argue that
\begin{align}\label{XvsYErgLim}
\mathbb E \big[|V(X;f)^n -  V(Y;f)^n  |]\to 0,\quad \text{as $n\to\infty$.}
\end{align}
To show \eqref{XvsYErgLim} we use that 
\begin{align}\label{slfjsdljfsdjfhsd}
 \mathbb E \big[|V(X;f)^n -  V(Y;f)^n  |]\leq \E[ |f(n^H\din X) - f(n^H\din Y)|]
\end{align}
which follows by the triangle inequality and stationarity of $\{(\din X, \din Y)\}_{i=k,\dots}$. 
From \cite[Eq.~(4.44)]{BasLacPod2016} we deduce that $\E[ |n^H\din X - n^H\din Y|^p] \to 0$ for all $p<\beta$, which by Lemma~\ref{dsklfhdskfhg} used on $p=1$ implies that the right-hand side of \eqref{slfjsdljfsdjfhsd} converges to zero.  This completes the proof of Theorem~\ref{th1}(ii).

\subsection{Proof of Theorem~\ref{th1}(iii)}


Let us first remark that the growth condition $|f(x)|\leq C(1\vee |x|^{q})$ for some $q$ with $q(k-\alpha)<1$ is weaker for larger $q$ and can therefore be thought of as 
\[|f(x)|\leq C |x|^{\frac 1 {k-\alpha}-\eps}\quad\text{ for }|x|\to\infty,\] if $k>\alpha$, whereas for $k\leq \alpha$ we require only that $f$ is of polynomial growth.
Since by the assumptions of the theorem we have $k-\alpha<1$, we may and do  assume that $q>1.$
We recall that a function $\xi:\R\to\R$ is absolutely continuous if there exists a locally integrable function $\xi'$ such that 
\[\xi(t)-\xi(s)=\int_s^t\xi'(u)\, du,\quad \text{for all $s<t.$}\]
This implies that $\xi$ is differentiable almost everywhere and the derivative coincides with $\xi'$ almost everywhere. If $\xi'$ can  be chosen absolutely continuous we say that $\xi$ is two times absolutely continuous, and similarly we define $k$-times absolute continuity.

By an application of \cite[Theorem 5.1]{BraSam1998} it has been shown in \cite[Lemma 4.3]{BasLacPod2016} that under the condition $(k-\alpha)(1\vee \beta)>1$ the process $X$ admits a $k$-times absolutely continuous version and the $k$-th derivative is a version of the process $(F_u)_{u\in\R}$ defined in \eqref{FDefGenVar}.
Moreover, \cite[Lemma 4.3]{BasLacPod2016} shows that for every $q\geq 1,q\neq \theta$ with $q(k-\alpha)<1$ the process $F$ admits a version with sample paths in $L^q([0,1])$, almost surely, which implies $\int_0^1 |f(F_u)|\,du<\infty.$
With these prerequisites at hand, Theorem \ref{th1}(iii) is a consequence of the following Lemma, which despite its intuitive statement requires some work. We denote by $W^{k,q}$ the space of $k$-times absolutely continuous functions $\xi$ on $[0,1]$ satisfying $\xi^{(k)}\in L^q([0,1]).$

\begin{lem}\label{fVar}
Let $\xi\in W^{k,q},$ and suppose that $f$ is continuous and $ |f(x)|\leq C(1\vee |x|^{q})$ for some $q\geq 1$. 
As $n\to \infty$ it holds that
\begin{equation}\label{dslfjsdghfsdf}
V(\xi;f,k)^n:=n^{-1}\sum_{i=k}^{n}f(n^k \din \xi)\to \int_0^1 f(\xi^{(k)}_s)\diff s.
\end{equation}
\end{lem}
\begin{proof}
Assume first $\xi\in \mathcal C^{k+1}([0,t])$. Taylor approximation shows that
\begin{align}\label{TaylorSmoothxi}
n^k\din\xi=\xi^{(k)}_{\frac{i-k}{n}}+a_{i,n},
\end{align}
where $|a_{i,n}|\leq C/n$ for all $n\geq 1, k\leq i\leq n$. We can therefore assume without loss of generality that $f$ has compact support and admits a concave modulus of continuity $\omega_f$, i.e.\ a continuous increasing function $\omega_f:[0,\infty)\to[0,\infty)$ with $\omega_f(0)=0$ such that $|f(x)-f(y)|\leq \omega_f(|x-y|)$ for all $x,y$.
We have by Jensen's inequality that
\begin{align*}
\limsup_{n\to\infty}\bigg| V(\xi,f,k)^n - \frac 1 n\sum_{i=k}^{n} f\big(\xi^{(k)}_{\frac{i-k}{n}}\big)\bigg|\leq \limsup_{n\to\infty}\bigg\{\omega_f\bigg(\frac 1 {n}\sum_{i=k}^{n} |a_{i,n}|\bigg)\bigg\}=0.
\end{align*}
The result follows by the convergence of Riemann sums
\[\frac 1 n\sum_{i=k}^{n} f\big(\xi^{(k)}_{\frac{i-k}{n}}\big)\to \int_0^1 f(\xi^{(k)}_s)\diff s.\]
In the following we extend the result to general $\xi\in W^{k,q}$ by approximating $\xi$ with a sequence $(\xi^{m})_{m\geq 1}$ of functions in $C^{k+1}([0,1])$. To this end, choose $\xi^m$ such that
\begin{align}
\displaystyle \int_0^1 |\xi^{(k)}_s-\xi_s^{m,(k)}|^q\diff s \leq 1/m,\quad \text{ for all $m$.}\label{xicon}
\end{align}
Indeed, the existence of such a sequence follows since continuous functions are dense in $L^q([0,1])$.
Note that \eqref{xicon} implies that $ \int_0^1 |\xi^{(k)}_s-\xi_s^{m,(k)}|\diff s \leq C/m^{1/q}$, since we assumed $q\geq 1$. 
Since $\xi^{m,(k)}$ converges in $L^q([0,1]),$ the family $(|\xi^{m,(k)}|^q)_{m\geq 1}$ is uniformly integrable. Hence, by the  assumption $|f(x)|\leq C(1\vee |x|^q)$  for $x\in \R$, we obtain uniform integrability of  $\{f(\xi^{m,(k)})_{m\geq 1}\}$. By continuity of $f$, we have that $f(\xi^{m,(k)})\to f(\xi^{(k)})$ in measure, and therefore also in $L^1([0,1])$:
\begin{align}\label{intcon}
\limsup_{m\to\infty}\int_{0}^1 |f(\xi^{(k)}_s)-f(\xi^{m,(k)}_s)|\diff s=0.
\end{align}
Hence, \eqref{dslfjsdghfsdf} follows if we show
\begin{equation}\label{DiffApprox}
\limsup_{m\to\infty}\sup_{n\in\N}|V(\xi;f,k)^n-V(\xi^m;f,k)^n|=0.
\end{equation}
  In order to show \eqref{DiffApprox} we split the sum \vspace{-1em}
\[|V(\xi;f,k)^n-V(\xi^m;f,k)^n|\leq\frac 1 n\sum_{i=k}^{n}\big|f(n^k\din \xi)-f(n^k\din \xi^m)\big|\]
into sums over the following sets of indices, where $N$ and $M$ are positive constants:
  \begin{align*}
A^{N}_n &=\{i\in\{k,...,n\}\,:\,n^k|\din \xi|> N\}\\
B^{N,M}_{m,n} &=\{i\in\{k,...,n\}\,:\,n^k|\din \xi|\leq N,\ n^k|\din \xi^m|> M \}\\
C^{N,M}_{m,n} &=\{i\in\{k,...,n\}\,:\,n^k|\din \xi|\leq N,\ n^k|\din \xi^m|\leq M \}.
\end{align*}
and estimate the corresponding sums separately. 
The following relationship between $\din\xi$ and $\xi^{(k)}$ will be essential.
For all $\xi\in W^{k,q}$ we have
\begin{align*}
\din \xi&= \int_{\frac{i-1}n}^{i/n}\int_{s_1-1/n}^{s_1}\dots\int_{s_{k-1}-1/n}^{s_{k-1}}\xi_{s_k}^{(k)}ds_k\dots ds_1.
\end{align*}
In particular, it follows that
\begin{equation}\label{k-thDerEst}
|n^k\din \xi|\leq \int_{[0,1]^k}n^k|\xi^{(k)}_{s_k}| \mathds 1_{\{(s_1,...,s_k)\in[(i-k)/n,i/n]^k\}}\diff s_k\dots \diff s_1= k^{k-1}\int_{\frac{i-k}n}^{i/n} n|\xi^{(k)}_s|\diff s.
\end{equation}
{\it The $A^{N}_n$ term:} We show that for given $\eps>0$ we can find sufficiently large $N$ such that
\begin{align}\label{A}
&\limsup_{m\to\infty}\sup_{n\in\N} \bigg\{n^{-1}\sum_{i\in A^{N}_n}\big|f(n^k \din \xi)-f(n^k \din \xi^m)\big|\bigg\} \nonumber\\
&\qquad \leq 
\limsup_{m\to\infty}\sup_{n\in\N} \bigg\{ n^{-1}\sum_{i\in A^{N}_n}|n^k \din \xi|^q+n^{-1}\sum_{i\in A^{N}_n}|n^k \din \xi^m|^q\mathds 1_{\{|n^k \din \xi^m|>1\}}\nonumber\\
&\phantom{\qquad \leq 
\limsup_{m\to\infty}\sup_{n\in\N} \bigg\{}+n^{-1}\sum_{i\in A^{N}_n}|f(n^k \din \xi^m)|\mathds 1_{\{|n^k \din \xi^m|\leq 1\}}\bigg\}\nonumber\\
&\qquad =:\limsup_{m\to\infty}\sup_{n\in\N} \big\{I_{1,n,N}+I_{2,n,m,N}+I_{3,n,m,N}\big\}\leq \eps,
\end{align}
First we consider $I_{1,n,N}$.
By \eqref{k-thDerEst} we have for all $i\in A^{N}_n$
\[N<k^{k-1}\int_{\frac{i-k}{n}}^{i/n}|\xi^{(k)}_s|n\diff s\leq k^{k-1}\int_{\frac{i-k}{n}}^{i/n}n|\xi^{(k)}_s|\mathds 1_{\{|\xi_s^{(k)}|>  C_{0,k}\}}\diff s+\frac{N}{2},\]
where $C_{0,k}:=N(2k^k)^{-1}$.
Therefore, again by \eqref{k-thDerEst}, it follows that
\begin{align}\label{A0}
|n^k\din \xi|\leq {}& k^{k-1} \int_{\frac{i-k}{n}}^{i/n}|\xi^{(k)}_s|n\diff s\nonumber\leq 2k^{k-1}\int_{\frac{i-k}{n}}^{i/n}|\xi^{(k)}_s|n\diff s-N \nonumber\\
\leq {}& 2k^{k-1}\int_{\frac{i-k}{n}}^{i/n}|\xi^{(k)}_s|\mathds 1_{\{|\xi_s^{(k)}|> C_{0,k}\}}n\diff s.
\end{align}
 Consequently, recalling that $q\geq 1$, we have by Jensen's inequality
\begin{align}\label{A1}
 n^{-1}\sum_{i\in A^{N}_n}|n^k\din \xi|^q
& \leq (2k^{k-1})^qk^{q-1} n^{-1}\sum_{i\in A^{N}_n}\int_{\frac{i-k}{n}}^{i/n}
|\xi^{(k)}_s|^q\mathds 1_{\{|\xi^{(k)}_s|>C_{0,k}\}}n\diff s\nonumber\\
& \leq (2k^{k})^q \int_{0}^{t} |\xi^{(k)}_s|^q\mathds 1_{\{|\xi^{(k)}_s|>C_{0,k}\}}\diff s.
\end{align}
It follows for sufficiently large $N>0$ that
\begin{align}\label{I_{1,n,N}}
\limsup_{m\to\infty}\sup_{n\in\N} \{I_{1,n,N}\}\leq \eps.
\end{align}
Next, we argue that the same holds for the $I_{2,n,m,N}$ term.
By \eqref{xicon} and Minkowski's inequality it follows for any $A\in\mathcal B([0,1])$ that $\int_A|\xi^{m,(k)}_s|^q\diff s\leq 2^{q-1} \int_A |\xi^{(k)}_s|^q\diff s +C/m$.
 Consequently, it holds  that
\begin{align}
  & n^{-1}\sum_{i\in A^{N}_n}|n^k\din \xi^m|^q\mathds 1_{\{|n^k \din \xi^m|>1\}}
 \leq C n^{-1}\sum_{i\in A^{N}_n}\int_{\frac{i-k}{n}}^{i/n}
|\xi^{m,(k)}_s|^q n\diff s\nonumber\\
& \qquad \leq C \sum_{i\in A^{N}_n} \int_{\frac{i-k}{n}}^{i/n}|\xi^{(k)}_s|^q\diff s+\frac C m\leq C \sum_{i\in A^{N}_n} \int_{\frac{i-k}{n}}^{i/n}|\xi^{(k)}_s|^q\mathds 1_{\{|\xi^{(k)}_s|>C_{0,k}\}}\diff s +\frac C m\nonumber\\
&\qquad \leq C  \int_{0}^{1}|\xi^{(k)}_s|^q\mathds 1_{\{|\xi^{(k)}_s|>C_{0,k}\}}\diff s +\frac C m,\nonumber
\end{align}
where the first inequality follows from \eqref{k-thDerEst} and the third from \eqref{A0}.
This shows that for sufficiently large $N$ it holds that
\begin{align}\label{I_{2,n,m,N}}
\limsup_{m\to\infty}\sup_{n\in\N} \{I_{2,n,m,N}\}\leq \eps.
\end{align}
Next, we estimate the term $I_{3,n,m,N}$. Introducing the notation
\[
D_{m,n} =\{i\in\{k,...,1\}\,:\,n^k|\din \xi^{(m)}|\leq 1\}
\]
we have
\begin{align}\label{I_{3,n,m,N}est}
I_{3,n,m,N}=n^{-1}\sum_{i\in A^{N}_n\cap D_{m,n}}|f(n^k\din\xi^{(m)})|\leq n^{-1}|A^{N}_n\cap D_{m,n}|\sup_{x\in (-1,1)}|f(x)|
\end{align}
where $|A^{N}_n\cap D_{m,n}|$ denotes the number of elements of $A^{N}_n\cap D_{m,n}$. Using \eqref{k-thDerEst} we have for all  $i\in A^{N}_n\cap D_{m,n}$
\[N-1\leq n^k|\din (\xi^{(k)}-\xi^{m,(k)})|\leq k^{k-1}\int_{\frac{i-k}n}^{i/n}|\xi_s^{(k)}-\xi_s^{m,(k)}|n\diff s,\]
and it follows that
\[|A^{N}_{n}\cap D_{m,n}|\leq \frac{n k^k}{N-1}\int_{0}^1|\xi_s^{(k)}-\xi_s^{m,(k)}|n\diff s\leq \frac {n k^kt}{(N-1)m^{1/q}},\]
where we recall \eqref{xicon}. With \eqref{I_{3,n,m,N}est} it follows that for all $N>1$ we have
\begin{align}\label{I_{3,n,m,N}}
\limsup_{m\to\infty}\sup_{n\in\N} \{I_{3,n,m,N}\}=0.
\end{align}
 Combining \eqref{I_{1,n,N}}, \eqref{I_{2,n,m,N}} and \eqref{I_{3,n,m,N}} we conclude that \eqref{A} holds for sufficiently large $N$.

\noindent 
{\it The $B^{N,M}_{m,n}$ term:}
We show that for any $\eps>0$ and any $N>0$ we can find a sufficiently large $M$ such that 
 \begin{align}\label{B}
&\limsup_{m\to\infty}\sup_{n\in\N} \bigg\{n^{-1}\sum_{i\in B^{N,M}_{m,n}}\big|f(n^k \din \xi)-f(n^k \din \xi^m)\big|\bigg\}\nonumber\\
 &\eqspace \leq \limsup_{m\to\infty}\sup_{n\in\N} \bigg\{n^{-1}\sum_{i\in B^{N,M}_{m,n}}\big|f(n^k \din \xi)|
+ n^{-1}\sum_{i\in B^{N,M}_{m,n}}|n^k \din \xi^m|^q\bigg\}\nonumber\\
&\eqspace=:\limsup_{m\to\infty}\sup_{n\in\N}\{J^1_{n,m,N,M}+J^2_{n,m,N,M}\}
<\eps.
\end{align}
The argument for $J^1_{n,m,N,M}$ is similar to the one used for $I_{3,m,n,N}$ above. We assume that $M>N.$ For $i\in B^{N,M}_{m,n}$ it holds by \eqref{k-thDerEst} that
\begin{align*}
M-N<n^k|\din (\xi-\xi^m)|\leq k^{k-1}n \int_{\frac{i-k}n}^{i/n}|\xi_s-\xi_s^m|\diff s.
\end{align*}
Consequently, we have for all $m\in\N$
\[ |B^{N,M}_{m,n}|\leq \frac{k^{k}n}{M-N}\int_0^1|\xi_s-\xi_s^m|\diff s \leq \frac{k^{k}n}{(M-N)m^{1/q}},\]
where $|B^{N,M}_{m,n}|$ denotes the number of elements in $B^{N,M}_{m,n}$. Then, it follows that for all $M>N$ 
\begin{align}\label{J^1_{n,m,N,M}}
\limsup_{m\to \infty}\sup_{n\in \N}\{J^1_{n,m,N,M}\}&\leq \limsup_{m\to\infty} \sup_{n\in \N}\{n^{-1}|B^{N,M}_{m,n}|\sup_{s\in[-N,N]}|f(s)|\}\nonumber\\
& \leq \limsup_{m\to\infty}\sup_{n\in N}\bigg\{\frac{k^{k}}{(M-N)m^{1/q}}\sup_{s\in[-N,N]}|f(s)|\bigg\}=0.
\end{align}
For $J^2_{n,m,N,M}$ we obtain by arguing as in \eqref{A1} with $\xi^{(k)}$ replaced by $\xi^{m,(k)}$ and $N$ replaced by $M$ that
\begin{equation}
 J^2_{n,m,N,M}
	\leq  (2k^{k})^q \int_{0}^{1}|\xi^{m,(k)}_s|^q\mathds 1_{\{|\xi^{m,(k)}_s|>M/2k^k\}}\diff s,
\end{equation}
for all $m,n, N.$
Since  $(|\xi^{m,(k)}|^q)_{m\geq 1}$ is uniformly integrable we can for $\eps>0$ find sufficiently large $M$ such that
\begin{equation}\label{J^2_{n,m,N,M}}
\limsup_{m\to \infty}\sup_{n\in \N}\{J^2_{n,m,N,M}\}\leq \eps.
\end{equation}
Now, \eqref{B} follows from \eqref{J^1_{n,m,N,M}} and \eqref{J^2_{n,m,N,M}}.

\noindent 
{\it The $C^{N,M}_{m,n}$ term:}
We show that for all $N,M>0$ we have that
\begin{align}\label{C}
\limsup_{m\to\infty}\sup_{n\in\N} \bigg\{n^{-1}\sum_{i\in C^{N,M}_{m,n}}\big|f(n^k \din \xi)-f(n^k \din \xi^m)\big|\bigg\} = 0.
\end{align}
Since $|n^k\din\xi|\leq N$ and $|n^k\din\xi^m|\leq M$ for all $i\in C^{N,M}_{m,n}$, we can replace $f$ by a continuous function $\wt \Phi_{N,M}$ with compact support, such that $f(x)=\wt \Phi_{N,M}(x)$ for all $x\in[-(N\vee M),N\vee M].$ Denote by $\wt\omega_{N,M}$ the concave modulus of continuity for $\wt \Phi_{N,M}$.
It holds that
\begin{align*}
&\sup_{n\in\N}\bigg\{n^{-1}\sum_{i\in C^{N,M}_{m,n}}\big|f(n^k \din \xi)-f(n^k \din \xi^m)\big|\bigg\}\\
	&\quad= \sup_{n\in\N}\bigg\{n^{-1}\sum_{i\in C^{N,M}_{m,n}}\big|\wt \Phi_{N,M}(n^k \din \xi)-\wt \Phi_{N,M}(n^k \din \xi^m)\big|\bigg\}\\
	&\quad\leq \sup_{n\in\N}\bigg\{\wt\omega_{N,M}\bigg(n^{-1}\sum_{j=k}^{n} n^k| \din \xi -\din\xi^m|\bigg)\bigg\}
\leq \wt\omega_{N,M}\bigg( k^k\int_0^1|  \xi^{(k)}_s-\xi_s^{m,(k)}|\diff s\bigg),
\end{align*}
where we used \eqref{k-thDerEst} in the last inequality. Now, \eqref{C} follows by \eqref{xicon}.

Finally, by \eqref{A}, \eqref{B} and \eqref{C}, for any $\eps>0$ we can  find sufficiently large $N,M$ such that
\[\limsup_{m\to\infty}\sup_{n\to\infty} \bigg(n^{-1}\sum_{i=k}^{n}\big|f(n^k \din \xi)-f(n^k \din \xi^m)\big|\bigg) < \eps.\]
 By letting $\eps\to 0$ we obtain \eqref{DiffApprox} and the proof of the lemma is complete.
\end{proof}


\section{Proofs of Theorems~\ref{lsjdlfjssss} and \ref{the2ndGenVar}}
\label{sec4}
\setcounter{equation}{0}
\renewcommand{\theequation}{\thesection.\arabic{equation}}


Before carrying out the proofs we will introduce some notation and estimates to be used 
in the following.

\noindent \emph{Definitions and notation:}
 For any function $\psi$ on the real line we denote  
\[D^k \psi(s):= \sum_{j=0}^{k}(-1)^j\binom k j \psi(s-j).\]
Furthermore, set 
\begin{equation}
g_n(s):= n^\alpha g(s/n),\qquad
\phi_t^n(s):= D^kg_n(t-s),\quad\text{and}\quad
Y_t^n := \int_{-\infty}^t \phi_t^n(s)\,dL_s,
\label{phiYDef}\end{equation}
for $n\in \N$.
By our assumptions on the function $g$ it holds that $g_n(s)\to s_+^\alpha$, and consequently $\phi^n_t(s)\to h_k(t-s)$ as $n\to\infty$, where $h_k$ was defined in \eqref{def-h-13}. Therefore, we complement \eqref{phiYDef} by defining\vspace{-1ex}
\begin{align*}
\phi_t^\infty(s):= h_k(t-s),\quad\text{and}\quad
Y_t^\infty := \int_{-\infty}^t h_k(t-s)\,dL_s.
\end{align*}
We recall that $(\F_t)_{t\in\R}$ denotes the filtration generated by $L$ and introduce additionally the $\sigma$-algebras
\[ \G^1_s:=\sigma (L_r-L_u\,|\,s\leq r,u\leq s+1),\]
remarking that $(\G^1_s)_{s\in\R}$ is not a filtration. 
We denote
\begin{align}\label{UDef}
 U_{j,r}^n := \int_{r}^{r+1} \phi_j^n(s)\,dL_s,\quad\text{where $n\in\N\cup\{\infty\}$ and $j\geq k$,}
 \end{align}
and introduce the notation
\begin{equation}\label{rhoDef}
\rho_j^n:=\rho_L \|\phi_j^n\|_{L^\beta(\R\setminus [0,1])},\quad\text{and}\quad \rho^n:=\rho_L\|\phi_1^n\|_{L^\beta(\R)}.
 \end{equation}
 Note that $Y^n_r\sim S\beta S(\rho^n)$ for all $r\geq k$ and $n\in \N$, which follows by \eqref{ScaParInt}.

\noindent
\emph{Preliminary estimates:} For $\xi<\beta$ and $\gamma>0$ there is a $C>0$ such that for all $\rho\in(0,1]$ and $S \sim S\beta S(1)$ we have 
\begin{align}
\E[|\rho S|^\xi\wedge |\rho S|^\gamma]&\leq 
\begin{cases}
C\rho^\beta \qquad & \text{for } \gamma>\beta,\\ 
C \rho^\gamma & \text{for }\gamma <\beta, 
\end{cases}
\label{StaVarEst}
\end{align}
where the first case follows by \cite[Lemma 5.5]{BasLacPod2016}, and the second case is a standard estimate.  
The function $\phi_j^n$ introduced above satisfies the estimate
\begin{align}\label{scaparest}
\|\phi_j^n\|_{L^\beta([0,1])}\leq C j^{\alpha-k},
\end{align}
for all $j\in\N$ and all $n\in\N\cup\{\infty\}$, which follows from Taylor expansion and the condition (A2) in Section \ref{sec2}.
Moreover, $\phi^n_j$  satisfies the following estimate  that has been derived in \cite[Eq.~(5.92)]{BasLacPod2016}. There exists a $C>0$ such that for all $n\in\N$ and $j\in\N$
\begin{align}\label{scaparjest}
\|\phi_j^n-\phi_j^\infty\|_{L^\beta([0,1])}\leq Cn^{-1}j^{\alpha-k+1}.
\end{align}

\begin{remark}\label{sdlfjslfj} \rm
In the proofs of Theorems~\ref{lsjdlfjssss} and \ref{the2ndGenVar} we may and do replace $\E[f(\rho_0 S)]$ by 
$\E[ f(n^H\Delta^n_{i,k} X)]$ in \eqref{clt}, \eqref{lsjdfgioskdow} and \eqref{slt}.  Indeed, to show this claim we first show that the function $\rho\mapsto G(\rho):=\E[ f(\rho S)]$ is 
continuously differentiable on $(0,\infty)$. Let $g_\beta$ denote the density of a S$\beta$S random variable. By substitution we have that 
\begin{equation}\label{sdlfjsgsss}
G(\rho) = \int_\R f(u) g_\beta(u/\rho)\,du.
\end{equation}
Since $\E[ |f(S)|]<\infty$ it follows that $\int |f(u)|(1\wedge |u|^{-1-\beta})\,du<\infty$, cf. \cite[Theorem 1.2]{Wat07}. 
We have that   $g_\beta\in C^\infty(\R)$, according to  \cite[Remark 28.2]{Sato}, and for all $r\geq 1$,  the $r$th derivative of $g_\beta$ satisfies
\begin{equation}\label{sdfsdf}
|g_\beta^{(r)}(x)|\leq C (1\wedge |x|^{-1-\beta-r}), \qquad x\in \R. 
\end{equation}
Indeed, to  show the estimate \eqref{sdfsdf} we use the dual representation  for  stable densities given in \cite[(2.5.5)]{Zolo1986}, which  implies  that 
\begin{equation}\label{ljsdlfhjsdlh}
g_\beta(x) = x^{-1-\beta}\tilde g(x^{-\beta}),\qquad x>0, 
\end{equation}
where $\tilde g$ is the density of a $1/\beta$-distribution. By $r$-times differentiation of \eqref{ljsdlfhjsdlh}, the estimate \eqref{sdfsdf} follows. Hence, from the estimate \eqref{sdfsdf} used on $r=1$ and \eqref{sdlfjsgsss},  it follows that $G\in C^1((0,\infty))$. 
By   \cite[Lemma 5.3]{BasLacPod2016} we have that 
\begin{equation}\label{dslfjsdlfjghs}
 \Big| n^H \rho_L \|g_{i,n}\|_{L^\beta(\R)}- \rho_0\Big| \leq 
 \begin{cases} C n^{-1} \qquad & \text{for }\alpha\in (0,k-2/\beta)\\
  C n^{(\alpha-k)\beta+1} \qquad & \text{for }\alpha\in (k-2/\beta,k-1/\beta). 
 \end{cases}
\end{equation}
Hence, for large enough $n$,
we obtain the estimate
\begin{equation}\label{sdlfjsdlhfs}
 \Big|\E[f(n^{H} \din X)] - \E[f(\rho_0 S)]\Big| \leq \Big(\max_{x\in [\rho_0 - \eps, \rho_0 + \eps]} |G'(x)| \Big)
\Big| n^H \rho_L \|g_{i,n}\|_{L^\beta(\R)}- \rho_0 \Big|,
\end{equation}
and by \eqref{sdlfjsdlhfs} and \eqref{dslfjsdlfjghs} it follows that 
\begin{equation}\label{sdlfjsldfjhss}
a_n \Big|\E[f(n^{H} \din X)] - \E[f(\rho_0 S)]\Big| \to 0\qquad \text{as }n\to \infty, 
\end{equation}
where $a_n= \sqrt{n}$ for Theorem~\ref{lsjdlfjssss}, $a_n= n^{k-\alpha-1/\beta}$ for Theorem~\ref{the2ndGenVar}(i), 
and $a_n = n^{1-\frac{1}{(k-\alpha)\beta}}$ for Theorem~\ref{the2ndGenVar}(ii). Eq.~\eqref{sdlfjsldfjhss} proves the above claim that we may replace $\E[f(\rho_0 S)]$ by 
$\E[ f(n^H\Delta^n_{i,k} X)]$ in Theorems~\ref{lsjdlfjssss} and \ref{the2ndGenVar}. 
%
%
%
%
%
 \qed
\end{remark}

By self-similarity of $L$, it holds that $\{n^H \Delta_{r,k}^n X\}_{r=k,...,n}\eqdist \{ Y^n_r\}_{r=k,...,n}$, and to deduce Theorems~\ref{lsjdlfjssss} and \ref{the2ndGenVar}  we show, cf.\ Remark~\ref{sdlfjslfj},  convergence in distribution for the properly normalised version of 
\begin{equation}\label{sdlfjsldfjhlsdhghdg}
S_n:=\sum_{r=k}^n \big(f(Y^n_r)-\E[f(Y^n_r)]\big)=\sum_{r=k}^n V_r^n,
\end{equation}
where we denoted $V_r^n:=f(Y^n_r)-\E[f(Y^n_r)]$ for brevity.


\subsection{Proof of Theorem~\ref{lsjdlfjssss}}


We recall the definition of $Y^n_r$ and $S_n$   from  \eqref{phiYDef} and \eqref{sdlfjsldfjhlsdhghdg}, and define additionally, for $a<b$, $a,b\in[0,\infty]$ and $m\geq 0$, 
\begin{align}
 Y^{n,[a,b]}_r={}&  \int_{r-b}^{r-a} \phi_r^n(s) \,dL_s,\qquad \qquad Y^{n,m}_r=Y^{n,[0,m]}_r,\\
 S_{n,m}= {}& \sum_{r=k}^n \big(f(Y^{n,m}_r)-\E[f(Y^{n,m}_r)]\big).
\end{align}
By \cite[Theorem 3.2]{Bil1999}, the statement of the theorem follows if we show the following three results
\begin{align}
&\lim_{m\to\infty} \limsup_{n\to\infty}\E[n^{-1}(S_n-S_{m,n})^2] =0,\label{S_{n,m}vsS_n}\\
&\frac 1 {\sqrt n} S_{n,m}\tol \mathcal N(0,\eta^2_m),\quad\text{ for some $\eta_m^2\in[0,\infty),$ and}\label{S_{n,m}Con}\\
&\eta^2_m\to \eta^2,\quad \text{as }m\to \infty. \label{etaCon}
\end{align} 
We show \eqref{S_{n,m}Con} first.
Set $\theta_j^{n,m}=\Cov(f(Y^{n,m}_k),f(Y^{n,m}_{k+j}))$ for $n\in\N\cup\{\infty\}$. 
Since the sequence $(Y_r^{n,m})_{r=k,...}$ is stationary 
the variance of $S_{n,m}$ is then given by
\[ n^{-1}\var(S_{n,m})=n^{-1}\bigg\{(n-k+1)\theta_0^{n,m}+2\sum_{j=1}^m (n-k-j)\theta_j^{n,m}\bigg\}.\] 
An application of  Lemma~\ref{dsklfhdskfhg} on $p=2$ yields 
that the covariances $\theta^{n,m}_j$ converge to $\theta^{\infty,m}_j$ for all $m,j$, as $n\to \infty$.
  Since the sequence $(Y^{n,m}_r)_{r=k,...}$ is $m$-dependent, \eqref{S_{n,m}Con} follows now from the central limit theorem for $m$-dependent sequences, see e.g.\ \cite{Ber1973}, with the limiting variance
\begin{equation}\label{eta_mDef}
\eta^2_m=\theta_0^{\infty,m}+2\sum_{j=1}^m \theta_j^{\infty,m}.
\end{equation}
Next, we argue that $\eta_m^2$ is a Cauchy sequence, which then shows \eqref{etaCon} with $\eta^2:=\lim_{m\to\infty}\eta_m^2$. This is indeed an immediate consequence of \eqref{S_{n,m}vsS_n} since
\begin{align*}
\Big||\eta_m|-|\eta_r|\Big|={}& \lim_{n\to\infty}n^{-1/2}\Big|\|S_{n,m}\|_{L^2}-\|S_{n,r}\|_{L^2}\Big|
\leq \limsup_{n\to\infty}n^{-1/2}\Big\|S_{n,m}-S_{n,r}\Big\|_{L^2}\\
\leq {}& \limsup_{n\to\infty}n^{-1/2}\Big\|S_{n,m}-S_{n}\Big\|_{L^2}+\limsup_{n\to\infty}n^{-1/2}\Big\|S_{n}-S_{n,r}\Big\|_{L^2}\to 0
\end{align*}
as $m,r\to \infty$ by \eqref{S_{n,m}vsS_n}. 
The proof of \eqref{clt} can thus be completed by deriving \eqref{S_{n,m}vsS_n}, which we do in the following.


We can express $S_n$ and $S_{n,m}$ as the telescoping sums
\begin{align}\label{ljsdfljsdflj}
S_n&=\sum_{r=k}^n \sum_{j=1}^\infty (\E[f(Y_r^n)|\G_{r-j+1}]-\E[f(Y_r^n)|\G_{r-j}]),\\
S_{n,m}&=\sum_{r=k}^n \sum_{j=1}^m (\E[f(Y_r^{n,m})|\G_{r-j+1}]-\E[f(Y_r^{n,m})|\G_{r-j}]).
\end{align}
Indeed, the first telescoping sum coincides with $S_n$ almost surely, since by the backwards martingale convergence theorem and Kolmogorov's 0-1 law it holds that $\E[f(Y_r^n)|\G_{r-j}]\toas \E[f(Y_r^n)]$, as $j\to\infty$.
We denote for $n\geq 1$ and $m,r,j\geq 0$
\begin{align*}
\xi^{n,m}_{r,j}=\E[f(Y_r^n)-f(Y_r^{n,m})|\G_{r-j+1}]-\E[f(Y_r^n)-f(Y_r^{n,m})|\G_{r-j}],
\end{align*}
and obtain
\begin{equation}\label{slfjsljf}
 S_n-S_{n,m}=\sum_{r=k}^n \sum_{j=1}^\infty \xi^{n,m}_{r,j}.
\end{equation}
Making the decomposition
\begin{align*}
&n^{-1}\E[(S_n-S_{n,m})^2]\\
&\leq 3n^{-1}\E\bigg[\bigg(\sum_{r=k}^n \sum_{j=m+1}^\infty \xi^{n,m}_{r,j}\bigg)^2\bigg]+3n^{-1}\E\bigg[\bigg(\sum_{r=k}^n \sum_{j=2}^{m} \xi^{n,m}_{r,j}\bigg)^2\bigg] +3n^{-1}\E\bigg[\bigg(\sum_{r=k}^n \xi^{n,m}_{r,1}\bigg)^2\bigg],
\end{align*}
we show that each summand on the right hand side converges to 0.
Observing that 
\[\Cov(\xi^{n,m}_{r,j},\xi^{n,m}_{r',j'})=0,\quad\text{unless }r-j=r'-j',\]
an application of Cauchy-Schwarz inequality and Fatou's lemma yields
\[ n^{-1}\E[(S_n-S_{n,m})^2]\leq 3n^{-1}Q_{n,1,m}+3n^{-1}Q_{n,2,m}+3n^{-1}Q_{n,3,m},\]
where
\begin{align*}
Q_{n,1,m}&=\sum_{r=k}^n\sum_{j=2}^m\sum^m_{j'=2} \E[(\xi^{n,m}_{r,j})^2]^{1/2}\E[(\xi^{n,m}_{r',j'})^2]^{1/2},\\
Q_{n,2,m}&=\sum_{r=k}^n\sum_{j=m+1}^\infty\sum^\infty_{j'=m+1} \E[(\xi^{n,m}_{r,j})^2]^{1/2} \E[(\xi^{n,m}_{r',j'})^2]^{1/2},\\
Q_{n,3,m}&=\sum_{r=k}^n \E[(\xi^{n,m}_{r,1})^2],
\end{align*}
and we denoted $r'=r-j+j'$. For the proof of \eqref{S_{n,m}vsS_n} it remains to show that 
\begin{align*}
\limsup_{n\to\infty}\frac 1 n Q_{n,i,m}\to 0,\quad \text{as }m\to\infty,\text{ for }i=1,2,3.
\end{align*}

{\it Estimation of $Q_{n,1,m}$:}
We introduce the notation
 \[\wt \Phi^n_j(x)=\E \big[f(x+Y^{n,j}_r)\big],\]
 which allows us to write $\E[f(Y^n_r)|\G_{r-j}]=\wt \Phi^n_j(Y_r^{n,[j,\infty]})$.
For $2\leq j\leq m$ we obtain
\begin{equation}\label{slfjshgshshs}
 \xi^{n,m}_{r,j}=\wt \Phi^n_{j-1}\big(Y_r^{n,[j-1,\infty]}\big)-\wt \Phi^n_{j}\big(Y_r^{n,[j,\infty]}\big)-\big\{\wt \Phi^n_{j-1}\big(Y_r^{n,[j-1,m]}\big)-\wt \Phi^n_{j}\big(Y_r^{n,[j,m]}\big)\big\}.
\end{equation}
The involved random variables can be decomposed into the sum of independent random variables as 
\begin{align*}
Y_r^{n,[j-1,\infty]}&=Y_r^{n,[j-1,j]}+Y_r^{n,[j,m]}+Y_r^{n,[m,\infty]}\\
Y_r^{n,[j,\infty]}&=Y_r^{n,[j,m]}+Y_r^{n,[m,\infty]}\\
Y_r^{n,[j-1,m]}&=Y_r^{n,[j-1,j]}+Y_r^{n,[j,m]}.
\end{align*}
Denoting by $F^n_{[j-1,j]},F^n_{[j,m]}$ and $F^n_{[m,\infty]}$ the corresponding distribution functions, we obtain
\begin{align*}
\E[(\xi^{n,m}_{r,j})^2]= &\int_\R\int_\R\int_\R \big\{\wt \Phi_{j-1}^n(u+v+w)-\wt \Phi_j^n(v+w)\\
&-\big(\wt \Phi_{j-1}^n(u+v)-\wt \Phi_j^n(v)\big)\big\}^2 dF^n_{[j-1,j]}(u)dF^n_{[j,m]}(v)dF^n_{[m,\infty]}(w).
\end{align*}
Using the relation $\wt \Phi_j^n(x)=\E f\big(x+Y_r^{n,j-1}+Y_r^{n,[j-1,j]}\big)=\int_\R \wt \Phi_{j-1}^n(x+z) dF^n_{[j-1,j]}(z),$
we obtain
\begin{align}\label{Ezetasquaredbound}
{}& \E[(\xi^{n,m}_{r,j})^2]=\int_\R\int_\R \int_\R \bigg(\int_\R D_{n,j}(u,v,w,z)dF^n_{[j-1,j]}(z)\bigg)^2dF^n_{[j-1,j]}(u)dF_{[j,m]}(v)dF^n_{[m,\infty]}(w)\nonumber\\
{}&\qquad  \leq \int_\R\int_\R \int_\R \int_\R D_{n,j}^2(u,v,w,z)dF^n_{[j-1,j]}(z)dF^n_{[j-1,j]}(u)dF^n_{[j,m]}(v)dF^n_{[m,\infty]}(w), 
\end{align}
where
\begin{align*}
D_{n,j}(u,v,w,z)&=\wt \Phi_{j-1}^n(u+v+w)-\wt \Phi_{j-1}^n(v+w+z)-\big(\wt \Phi_{j-1}^n(u+v)-\wt \Phi_{j-1}^n(v+z)\big)\\
&=\Phi_{\rho_{j-1}^n}(u+v+w)-\Phi_{\rho_{j-1}^n}(v+w+z)-\big( \Phi_{\rho_{j-1}^n}(u+v)- \Phi_{\rho_{j-1}^n}(v+z)\big),
\end{align*}
and  
 $\rho^n_{j-1}$ is the scale parameter of the \sbs\ random variable $Y^{n,j-1}_r$.
It follows from Lemma~\ref{DerBou} that $D_{n,j}$ satisfies the estimate 
\begin{align}\label{Dsquaredest}
D^2_{n,j}(u,v,w,z)\leq C \big(|u-z|^{2p}\wedge (u-z)^2\big)(|w|^{2p}\wedge w^2),\quad\text{for all $j\geq 2,n\in\N$,}\qquad 
\end{align}
where $p$ is as in \eqref{Est2}, provided $\{\rho_{j-1}^n\}_{j\geq 2,n\in\N}$ is bounded away from 0 and $\infty$. This is indeed the case, as 
follows from the estimates
\begin{align*}
(\rho^n_{j-1})^\beta&=  \int_{r-j+1}^r |\phi_r^n(s)|^\beta ds \ \leq\ \|\phi_r^n\|_{L^\beta(\R)}^\beta\ =\ \rho^n,\quad\text{and}\\
(\rho^n_{j-1})^\beta &\geq \int_{r-1}^r |\phi_r^n(s)|^\beta ds \ \to\  \int_0^1 s^{\alpha\beta}ds\ >\ 0\quad\text{ as }n\to\infty,
\end{align*}
where the convergence follows by the dominated convergence theorem, since Assumption~(A) implies the existence of a $C>0$ such that $|\phi_r^n(s)|\leq C |r-s|^\alpha$ for all $s\in[r-1,r]$ and all $n\geq 1$.

Applying \eqref{Dsquaredest} on the right hand side of \eqref{Ezetasquaredbound} yields the estimate
 \begin{align*}
& \E[(\xi^{n,m}_{r,j})^2]\\ &\quad  \leq C\bigg(\int_{\R^2}|u-z|^{2p}\wedge(u-z)^2\, dF^n_{[j-1,j]}(u)dF^n_{[j-1,j]}(z)\bigg)\int_{\R}|w|^{2p}\wedge w^2\, dF^n_{[m,\infty]}(w).
\end{align*}
It follows now from \eqref{scaparest} and \eqref{StaVarEst} that $\E[(\xi^{n,m}_{r,j})^2]\leq C(\rho^n_{[j-1,j]}\rho^n_{[m,\infty]})^\beta,$ where $\rho^n_{[j-1,j]}$ and $\rho^n_{[m,\infty]}$ are the scale parameters of the stable distributions $F^n_{[j-1,j]}$ and $F^n_{[m,\infty]}$, respectively.
By \eqref{ScaParInt} and \eqref{scaparest} the scale parameters satisfy $\rho^n_{[j-1,j]}=\rho_L \|\phi_j^n\|_{L^\beta([0,1])}\leq C j^{\alpha-k},$ and 
\begin{align}\label{rho^n_{[m,infty]}Est}
(\rho^n_{[m,\infty]})^\beta= \rho_L \int_{-\infty}^{r-m} |\phi_r^n(s)|^\beta ds =\rho_L \sum_{l=m+1}^\infty \|\phi_l^n\|^\beta_{L^\beta([0,1])}\leq C\sum_{l=m+1}^\infty l^{ \beta (\alpha-k)}.
\end{align}
It follows that
\[\E[(\xi^{n,m}_{r,j})^2]\leq Cj^{\beta(\alpha-k)}\sum_{l=m+1}^\infty l^{ \beta (\alpha-k)},\]
for all $j\in \{2,...,m\}$ and we obtain
\[\limsup_{n\to\infty}\frac 1 n Q_{n,1,m}\leq C \bigg(\sum_{j=2}^m j^{\frac \beta 2(\alpha-k)}\bigg)^2\bigg(\sum_{l=m+1}^\infty l^{ \beta (\alpha-k)}\bigg),\]
which converges to 0, as $m\to\infty$ since $\beta(\alpha-k)<-2.$

{\it Estimation of $Q_{n,2,m}$:}
This term is estimated by similar, and in fact easier, arguments as used for the estimation of $Q_{n,1,m}$ which we do not  repeat here.

{\it Estimation of $Q_{n,3,m}$:}
Using the inequality $\E\big\{\E[X|\G]-\E[Y|\mathcal F]\big\}^2\leq 2 \E X^2+2 \E Y^2$ we obtain
\[\frac 1 n Q_{n,3,m}\leq \frac 4 n\sum_{r=k}^n \E[(f(Y^n_r)-f(Y^{n,m}_r))^2] = \frac{n-k+1}n\E[(f(Y^n_1)-f(Y^{n,m}_1))^2],\]
and it is sufficient to argue that $\limsup_{n\to\infty} \E[(f(Y^n_1)-f(Y^{n,m}_1))^2]\to 0$ as $m\to\infty$. However, this follows by   Lemma~\ref{dsklfhdskfhg} with $p=2$, and completes the proof of \eqref{S_{n,m}vsS_n}, and thus of 
Theorem~\ref{lsjdlfjssss}.
\qed

\subsection{Proof of Theorem~\ref{the2ndGenVar}(i)}

In the following section we set for all $n\in \N$
\begin{equation}
\tilde S_n  = \Phi_{\rho^n}'(0) \sum_{r=k}^n Y^n_r, \qquad \text{where}\quad \Phi_\rho'(x)=\frac{\partial }{\partial x}\Phi_\rho(x).
\end{equation}
To prove Theorem~\ref{the2ndGenVar}(i), it is enough to show that the following \eqref{con-ssss-1} and \eqref{con-ssss-2}  hold, where 
\begin{align}\label{con-ssss-1}
{}& n^{k-\alpha-1/\beta-1}(S_n-\tilde S_n)\toop 0,\\ 
\label{con-ssss-2}{}& n^{k-\alpha-1/\beta-1}\tilde S_n\tol S\beta S(\sigma),
\end{align}
and  $\sigma$ is given in \eqref{ljsdfhhjljeri}. 

\noindent \emph{Proof of \eqref{con-ssss-1}}:
We let   $ \tilde f_\rho(x) = f(x) - \Phi'_\rho(0) x$,  and set 
\begin{equation}
\tilde \Phi_{\rho}(x) := \E[ \tilde f_\rho(x+S)] - \E[ \tilde f_\rho(S)] = \Phi_\rho(x)- \Phi'_\rho(0) x,
\end{equation}
 for all $x\in \R$,  $n\in \N$ and  $S\sim S\beta S(\rho)$. For all $\epsilon >0$ there exists $C>0$ such that $|\tilde \Phi_\rho'(x)|\leq C$ and $|\tilde \Phi_\rho''(x)|\leq C$ for all $x\in \R$ and $\rho\in [\epsilon, \epsilon^{-1}]$, and since
 $\tilde \Phi_\rho(0)=\tilde \Phi'_\rho(0)=0$ it follows that 
 \begin{equation}\label{dslfjsldhfoshdf}
 |\tilde \Phi_\rho(x)|\leq C_\epsilon (|x|\wedge |x|^2),
 \end{equation}
     which  will be crucial for the following estimates.  We set 
\begin{align}
\zeta^n_{r,j}{}& = \E[\tilde f_{\rho^n} (Y_r^n) |\G_{r-j+1}]-\E[\tilde f_{\rho^n}(Y_r^n)|\G_{r-j}]-
\E[\tilde f_{\rho^n}(Y_r^n) |\G_{r-j}^1]+\E[f(Y_r^n)], 
\end{align}
and decompose $S_n-\tilde S_n$ as follows
\begin{equation}\label{aljdfljsdf}
S_n-\tilde S_n = \sum_{r=k}^n \Big( \sum_{j=1}^\infty \zeta^n_{r,j} \Big) + \sum_{r=k}^n \Big( \sum_{j=1}^\infty \big(\E[\tilde f_{\rho^n}(Y_r^n) |\G_{r-j}^1] - \E[f(Y_r^n)]\big)\Big)=: V_n+W_n. 
\end{equation}
In the following we will estimate $W_n$ and $V_n$ separately. 

\noindent \emph{Estimation of $W_n$:}
By the  substitution $s=r-j$ we obtain the representation 
\begin{align}
W_n ={}&  \sum_{s=-\infty}^{n-1} \Big( \sum_{j=(k-s)\vee 1}^{n-s} \big(\E[\tilde f_{\rho^n}(Y_{s+j}^n) |\G_{s}^1] - \E[f(Y_r^n)]\big)\Big)=\sum_{s=-\infty}^{n-1} D^n_s, \qquad \text{where}\\
D^n_s:={}& \sum_{j=(k-s)\vee 1}^{n-s} \big(\E[\tilde f_{\rho^n}(Y_{s+j}^n) |\G_{s}^1] - \E[f(Y_{s+j}^n)]\big).
\end{align}
Since $\{D^n_s:s\in \Z\}$ is a martingale difference sequence, the von Bahr--Esseen inequality \cite[Theorem~1]{BahEss1965} yields  
that for any $\gamma\in (1,\beta)$ 
\begin{align}\label{sdflkjsdghd}
\E[ |W_n|^\gamma] \leq 2\sum_{s=-\infty}^{n-1}  \E[ | D^n_s|^\gamma] \leq 2\sum_{s=-\infty}^{n-1} \Big( \sum_{j=(k-s)\vee 1}^{n-s} \big(\E\big[ \big|\E[\tilde f_{\rho^n}(Y_{s+j}^n) |\G_{s}^1] - \E[ f(Y_{s+j}^n)]\big|^\gamma\big]\big)^{1/\gamma}\Big)^\gamma, 
\end{align}
where the second inequality follows by Minkowski's inequality. 
We have that 
\begin{align}
 |\Phi'_{\rho^n}(0) - \Phi'_{\rho^n_j}(0)|\leq C | \rho^n - \rho^n_j |  \leq C \Big|   |\rho^n|^\beta - |\rho^n_j|^\beta \Big| = C \| \phi^n_j \|_{L^\beta([0,1])}^\beta\leq C j^{\beta (\alpha-k)} \qquad \label{sdfkljsdflhs}
\end{align}
where the first inequality follows by boundedness of $\frac{\partial^2}{\partial \rho\partial x}\Phi_\rho(x)$, for the second inequality we use that $\rho^n, \rho^n_j$ are bounded away from 0 and $\infty$, cf.\ Lemma~\ref{lemrhobou}, and the last inequality is \eqref{scaparest}.
By a calculation similar to \eqref{slfjshgshshs} we obtain the identity   
\begin{equation}
\E[\tilde f_{\rho^n}(Y_{s+j}^n)|\G_{s}^1] - \E[ f(Y_{s+j}^n)] =  \Phi_{\rho^n_j}(U_{j+s,s}^n)- \E[  \Phi_{\rho^n_j}(U_{j+s,s}^n)]-  \Phi'_{\rho^n}(0)U^n_{j+s,s},
\end{equation}
and hence for all $r\in (1,2)$ with $r\gamma <\beta$, we have
\begin{align}
 &  \E\Big[ \Big|\E[\tilde f_{\rho^n}(Y_{s+j}^n)|\G_{s}^1] - \E[ f(Y_{s+j}^n)]\Big|^\gamma\Big] \\
 &\eqspace 
 \leq\ C \Big( \E[ | \tilde \Phi_{\rho^n_j}(U_{j+s,s}^n)|^\gamma] + |\Phi_{\rho^n}'(0)-\Phi'_{\rho_j^n}(0)|^\gamma \E[|U^n_{s+j,s}|^\gamma] \Big) \\
&\eqspace \leq\ C \Big(  \E [ | U_{j+s,s}^n |^{r \gamma}] + |\Phi_{\rho^n}'(0)-\Phi'_{\rho_j^n}(0)|^\gamma j^{\gamma(\alpha-k)} \Big)
\ \leq\  C \Big(  j^{\gamma r (\alpha-k)}+  j^{\gamma(\alpha-k)(1+\beta)} \Big)\\[-1em]
&\eqspace \leq\ C   j^{\gamma r (\alpha-k)}, 
 \label{sldfjshhshs} 
\end{align}
where the estimate  $|\tilde \Phi_{\rho^n_j}(x)|\leq C |x|^r$  is used in the second inequality (cf.\ \eqref{dslfjsldhfoshdf}), and \eqref{sdfkljsdflhs} is used in the third inequality. 
From  \eqref{sdflkjsdghd} and \eqref{sldfjshhshs} we deduce  
\begin{align}
 {}& \E[ |W_n|^\gamma] \leq  C\sum_{s=-\infty}^{n-1} \Big( \sum_{j= (k-s)\vee 1}^{n-s} j^{r(\alpha-k)}  \Big)^\gamma \\ 
{}& \quad =  C\Big(\sum_{s=-\infty}^{-n}  \Big( \sum_{j= k-s }^{n-s} j^{r(\alpha-k)}  \Big)^\gamma+\sum_{s=-n+1}^{k-2}  \Big( \sum_{j= k-s }^{n-s} j^{r(\alpha-k)}  \Big)^\gamma+\sum_{s=k-1}^n  \Big( \sum_{j= 1}^{n-s} j^{r(\alpha-k)}  \Big)^\gamma  \Big)\\ {}& \quad =: C\big( A_n'+A_n''+A_n'''\big). 
\end{align}
We may and do choose $r$ and $\beta$ such that $r(\alpha-k)\neq -1$ and $-\beta<r\gamma (\alpha-k)<-1$. Recall that $-\beta<\beta(\alpha-k)<-1$ by assumption, and $r,\gamma>1$ satisfies $r \gamma <\beta$.  We start by estimating $A_n'$ as follows 
\begin{equation}\label{est-a1}
A_n'\leq \sum_{s=n}^\infty (n s^{r(\alpha-k)})^\gamma \leq C n^{\gamma r (\alpha-k)+1+\gamma}
\end{equation}
where we have used $r\gamma(\alpha-k)<-1$ in the last inequality. 
By Jensen's inequality we have  \noeqref{est-a2}
\begin{equation}\label{est-a2}
A_n''\leq n^{\gamma -1} \sum_{s=1}^n   \Big( \sum_{j=k+s}^{n+s} j^{r \gamma (\alpha-k)} \Big) \leq C n^{\gamma-1} 
\sum_{s=1}^n s^{r\gamma (\alpha-k) +1}\leq C n^{r\gamma(\alpha-k)+\gamma+1}, 
\end{equation}
where we have used $r\gamma (\alpha-k)<-1$ in the second inequality, and $r\gamma (\alpha-k)>-2$ in the last inequality.  
For $\gamma<\beta$ close enough to $\beta$ we have that $r(\alpha-k)>-1$ for all $r\in (1,\beta/\gamma)$, by the assumption $\alpha>k-1$. The substitution $v=n-s$ yields that 
\begin{align}\label{est-a3}
A_n''' \leq C \sum_{v=1}^n \Big( \sum_{j= 1}^{v} j^{r(\alpha-k)}  \Big)^\gamma \leq 
C \sum_{v=1}^n v^{\gamma r(\alpha-k)+\gamma} \leq n^{\gamma r(\alpha-k)+\gamma+1}, 
\end{align}
where we  have used  $r(\alpha-k)>-1$ in the second inequality, and $\gamma r(\alpha-k)>-2$ in the last inequality.  
The above three estimates \eqref{est-a1}--\eqref{est-a3} show the bound 
\begin{equation}\label{sdlfjhsss}
 \E[ |W_n|^\gamma] \leq  C n^{\gamma r(\alpha-k)+\gamma+1}.
\end{equation} 

\noindent \emph{Estimation of $V_n$:}
By    the substitution $s=r-j$ we have that 
\begin{equation}
 V_n= \sum_{r=k}^n \Big( \sum_{j=1}^\infty \zeta^n_{r, j}\Big) =\sum_{s=-\infty}^{n-1} \Big( \sum_{r=k\vee (s+1)}^n \zeta^n_{r,r-s}\Big) = \sum_{s=-\infty}^{n-1} M^n_s, 
\end{equation}
where $M^n_s := \sum_{r=k\vee (s+1)}^n \zeta^n_{r,r-s}$.
 Since $(M_s^n)_{s\in \Z}$ is a martingale difference for all fixed $n\in \N$, we have by the von Bahr--Esseen inequality \cite[Theorem~1]{BahEss1965}  for all $\gamma\in [1,2]$ with $\gamma<\beta$ that   
\begin{align}\label{sldfjshghgg}
 \E[ |V_n|^\gamma] \leq  2 \sum_{s=-\infty}^{n-1} \E[ |M^n_s|^\gamma]
 \leq  \sum_{s=-\infty}^{n-1} \Big( \sum_{r=k\vee (s+1)}^n \|\zeta^n_{r,r-s}\|_\gamma \Big)^\gamma,
\end{align}
where  the last inequality follows from the Minkowski inequality.   
In the following we define the random variables  $\vartheta_{r,j,l}^n$, $l\geq j$,  by
\begin{align}\label{thetadefslt}
\vartheta_{r,j,l}^n 
={}& \E[\zeta_{r,j}^n\,|\,\G_{r-j}^1\vee \G_{r-l}]-\E[\zeta_{r,j}^n\,|\,\G_{r-j}^1\vee \G_{r-l-1}]\\
={}& \E[f(Y_r^n)\,|\,\G^1_{r-j}\vee\G_{r-l}]-\E[f(Y^n_r)\,|\,\G^1_{r-j}\vee \G_{r-l-1}]\nonumber\\
{}&  -\big\{\E\big[\E[f(Y_r^n)\,|\,\G_{r-j}]  \,|\,\G_{r-j}^1\vee \G_{r-l}\big]-\E\big[\E[f(Y_r^n)\,|\,\G_{r-j}]  \,|\,\G_{r-j}^1\vee \G_{r-l-1}\big]\big\}.\nonumber
\end{align}
By a  telescoping sum argument similar to \eqref{slfjsljf}, we obtain the representation 
\begin{equation}
\zeta^n_{r,j} = \sum_{l=j}^\infty \vartheta^n_{r,j,l}.
\end{equation}
Since $\{\vartheta_{r,j,l}^n \!: l = j, 2, \dots\}$ is a martingale difference sequence,  the von Bahr--Esseen inequality \cite[Theorem~1]{BahEss1965} yields that  
\begin{equation}\label{lhsgggsusw}
 \E[ | \zeta^n_{r,j}|^\gamma]\leq 2 \sum_{l=j}^\infty \E[ |\vartheta^n_{r,j,l}|^\gamma] 
 \leq C  \sum_{l=j}^\infty j^{(\alpha-k)\gamma} l^{(\alpha-k)\gamma}\leq C j^{2(\alpha-k)\gamma+1} 
\end{equation}
for all $\gamma\in (1,\beta)$ such that $(\alpha-k)\gamma<-1$, where we have used Lemma~\ref{lemthetaest} in the second inequality.  From \eqref{sldfjshghgg} and \eqref{lhsgggsusw}  we have  
\begin{align}
{}&\E[ |V_n|^\gamma] \leq  C \sum_{s=-\infty}^{n-1}\Big( \sum_{r=k\vee (s+1)}^n (r-s)^{2(\alpha-k)+1/\gamma} \Big)^\gamma 
\\ {}& \quad  =  C\Big\{\sum_{s=-\infty}^{-n}\Big( \sum_{r=k}^n (r-s)^{2(\alpha-k)+1/\gamma} \Big)^\gamma +\sum_{s=-n+1}^{k-1}\Big( \sum_{r=k}^n (r-s)^{2(\alpha-k)+1/\gamma} \Big)^\gamma \\ {}& \quad \phantom{ =  C\Big\{} +\sum_{s=k}^{n-1}\Big( \sum_{r= s+1}^n (r-s)^{2(\alpha-k)+1/\gamma} \Big)^\gamma \Big\} 
=:\{B_n'+B_n''+B_n'''\}.
\end{align}
We estimate $B_1', B_n''$ and $B_n'''$ in a similar fashion  as in \eqref{est-a1}--\eqref{est-a3}, but need to  divide into 
several  cases depending on the value of $\gamma (\alpha-k)$. We arrive with the following estimates 
\begin{align}
B_n'\leq {}&  C n^{2 \gamma (\alpha-k)+\gamma+2}, \qquad \qquad B_n''\leq  
\begin{cases}C n^{2 \gamma (\alpha-k)+\gamma+2} \qquad  & \text{for } \gamma(\alpha-k)>-3/2,\\ C n^{\gamma-1}  & \text{for }\gamma(\alpha-k)<-3/2,\end{cases} \\ 
B_n'''\leq {}& \begin{cases} Cn^{2 \gamma (\alpha-k)+\gamma+2}\qquad    & \text{for }\gamma(\alpha-k)>(-1-\gamma)/2, \\ C n& \text{for }\gamma(\alpha-k)<(-1-\gamma)/2, \end{cases}
\end{align}
which implies  
\begin{equation}\label{sdlfhsdfhs}
 \E[ |V_n|^\gamma] \leq  C \Big(n^{2 \gamma (\alpha-k)+\gamma+2}+n\Big). 
\end{equation}
Combining the   estimates \eqref{aljdfljsdf}, \eqref{sdlfjhsss} and \eqref{sdlfhsdfhs} yields   
\begin{equation}\label{sldfjlsdhhsdh}
  \E\Big[\Big| n^{k-\alpha-1/\beta-1}(S_n-\tilde S_n)\Big|^\gamma\Big] \leq C \Big( n^{-\gamma/\beta+\gamma (r-1)(\alpha-k)+1}+ n^{-\gamma/\beta+ \gamma (\alpha-k)+2}+n^{\gamma(k-\alpha-1/\beta - 1)+1} \Big).
\end{equation}
The  three terms on the right-hand side of \eqref{sldfjlsdhhsdh} converge to zero as $n\to\infty$. Indeed, it follows that the first term converges to zero, by  choosing   $\gamma\in (1,\beta)$ close enough to $\beta$  and  then choose
$r\in (1,\beta/\gamma)$ close enough to $\beta/\gamma$, which can be done under the above restrictions on $r$ and $\gamma$. The second term converges to zero due to the assumption $\gamma(\alpha-k)<-1$ and the third term converges to 0 for $\gamma$ close enough to $\beta$ by the assumption $\alpha>k-1$. Hence, \eqref{sldfjlsdhhsdh} completes the proof of \eqref{con-ssss-1}. 

\noindent \emph{Proof of \eqref{con-ssss-2}}:  
In the following  we  write $g_{i,n, k}$ for $g_{i,n}$, given in  \eqref{gindef}, to stress the dependence of the order of increments $k\geq 1$. We have 
\begin{equation}\label{ljshdfss}
n^{k-\alpha-1/\beta-1}\tilde S_n\eqdist \Phi'_{\rho^n}(0) n^{k-1}\sum_{r=k}^n \Delta^n_{r,k} X = \Phi'_{\rho^n}(0) n^{k-1} \Big( \Delta^{n}_{n,k-1} X - \Delta^{n}_{k-1,k-1}X\Big),
\end{equation}
where the last equality follows by the telescoping sum structure. 
According to  the mean value theorem there exists  $\theta_1,\theta_2\in [-k/n,0]$ (depending on $n$ and $s$)  such that 
\begin{align}
  \Big| n^{k-1} \Big( g_{n,n, k-1}(s) - g_{k-1,n, k-1}(s)\Big)\Big|  {}& \leq C \Big|  g^{(k-1)}(1-s+\theta_1) -g^{(k-1)}(-s+\theta_2)\Big| \\ {}& 
\leq C \Big( \1_{\{|s|\leq 1\}}+ \1_{\{s<-1\}} |s|^{\alpha-k}  \Big)=:c(s),
\label{lsjdflhsh} 
\end{align}
where the last inequality follows by Assumption (A2) and the mean value theorem for $s<-1$, and by the assumption $\alpha>k-1$ for the case $|s|\leq 1$.
The function $c$ in  \eqref{lsjdflhsh} is in $L^\beta(ds)$, due to the fact that 
$\alpha<k-1/\beta$.
Hence, by the dominated convergence theorem, we have 
\begin{equation}\label{lsdjfhhss}
\int_\R \Big| n^{k-1} \Big( g_{n,n, k-1}(s) - g_{k-1,n, k-1}(s)\Big) \Big|^\beta\,ds\to \int_\R \Big| g^{(k-1)}(1-s)- g^{(k-1)}(-s)\Big|^\beta\,ds =: c_0<\infty. 
\end{equation}
as $n\to \infty$.
By \cite[Lemma~5.3]{BasLacPod2016}, $\rho^n\to \rho^\infty$ which implies that $\Phi'_{\rho^n}(0)\to \Phi'_{\rho^\infty}(0)$ by continuity of $\rho\mapsto \Phi'_\rho(0)$ on $(0,\infty)$. 
 Therefore, by   \eqref{ljshdfss} and \eqref{lsdjfhhss} we conclude that 
\begin{equation}\label{ljsdfhhjljeri}
n^{k-\alpha-1/\beta-1}\tilde S_n \schw S\beta S(\sigma)\qquad \textrm{with } \sigma := \rho_L \Phi'_{\rho^\infty}(0)c_0^{1/\beta},
\end{equation}
which completes the proof of Theorem~\ref{the2ndGenVar}(i).

\subsection{Proof of Theorem~\ref{the2ndGenVar}(ii)}


Before we start the proof of Theorem~\ref{the2ndGenVar}(ii) we will deduce some estimates on $\Phi_\rho(x)$ relying on the assumption of Appell rank greater or equal 2 in this theorem. 
Let $\epsilon\in (0,1)$ be fixed. The mean value theorem, together with assumptions \eqref{Est2} and  \eqref{Est3} and  the  Appell rank greater or equal two condition, $\frac{\partial }{\partial x}\Phi_\rho(0)=0$ for all $\rho>0$,   implies that 
\begin{equation}\label{klsjhdfkhsdklfhhd}
|\Phi_\rho(x)-\Phi_\rho(y)| \leq C\Big( (1\wedge |x| +1\wedge |y|)|x-y|\1_{\{|x-y|\leq 1\}}+|x-y|^p\1_{\{|x-y|>1\}}\Big)
\end{equation}
for all $x, y\in \R$ and $\rho\in [\epsilon,\epsilon^{-1}]$. Specializing  \eqref{klsjhdfkhsdklfhhd} to $y=0$  yields that \begin{equation}\label{sljsdfljs}
|\Phi_\rho(x)|\leq C (|x|^p\wedge |x|^2), \qquad x\in \R,\  \rho\in [\epsilon,\epsilon^{-1}].
\end{equation}
 Next let $x\in \R$ and $\rho_1, \rho_2\in [\epsilon,\epsilon^{-1}]$.  
From an application of the mean value theorem in the $\rho$ variable it follows that there exists $\tilde \rho\in [\epsilon,\epsilon^{-1}]$ such that 
\begin{align}\label{ljsdlfhskldhfsd}
 |\Phi_{\rho_1}(x)-\Phi_{\rho_2}(x)| \leq {}&C | \rho_1-\rho_2| \cdot\Big|\frac{\partial }{\partial \rho}\Phi_{\tilde \rho}(x)\Big| 
\\   \leq {}&C  | \rho_1-\rho_2| \Big(1\wedge |x|^2\Big)\leq C  | \rho_1-\rho_2| \Big(|x|^p\wedge |x|^2\Big)
\end{align}
where in the second inequality we  use  that    $|\frac{\partial^3}{\partial x^2\partial \rho}\Phi_{\rho}(x)|\leq C$,  $|\frac{\partial}{\partial \rho}\Phi_{\rho}(x)|\leq C$,  $\frac{\partial^2}{\partial x\partial \rho}\Phi_{\rho}(0)=0$ and $\frac{\partial}{\partial \rho}\Phi_{\rho}(0)=0$;  the latter fact follows since $\Phi_\rho(0)=0$ for all $\rho>0$.

For all $r\geq k$ we define $Z_r$ by 
\begin{align}\label{ZDef}
Z_r:=\sum_{j=1}^\infty \big\{ \Phi_{\rho_j^\infty} (U_{j+r,r}^\infty)-\E[\Phi_{\rho_j^\infty} (U_{j+r,r}^\infty)]\big\},
\end{align}
where the sum is almost surely absolutely convergent. Indeed, this fact follows by the same arguments as in  \cite[(5.19) b]{BasLacPod2016}, where this statement is derived in the context of power variation (the proof relies on 
the estimate \eqref{sljsdfljs}).   Since for all $j\geq 0$ the sequence $(U^\infty_{j+r,r})_{r\geq k}$ is i.i.d., the random variables $Z_r, r\geq k$ are i.i.d.\,as well. 
For $n\geq 1, m, r\geq 0$ we denote
\begin{align}
\zeta_{r,j}^n&:=\E[V_r^n|\mathcal F_{r-j+1}]-\E[V_r^n|\mathcal F_{r-j}]-\E[V_r^n|\mathcal F^1_{r-j}],\nonumber\\
R_r^n &:=\sum_{j=1}^\infty \zeta_{r,j}^n\qquad\text{and}\qquad Q_r^n:= \sum_{j=1}^\infty \E[V_r^n\,|\, \G_{r-j}^1].\label{RQDefslt}
\end{align}
The sums $R_r^n$ and $Q_r^n$ converge almost surely, which follows by the arguments of  \cite[(5.21)]{BasLacPod2016} and thereafter. 
We obtain the following important decomposition 
\begin{align}\label{maindecomp}
S_n=\sum_{r=k}^n R_r^n + \sum_{r=k}^n (Q_r^n-Z_r)+\sum_{r=k}^n Z_r,
\end{align}
where we will argue  that the first two sums in \eqref{maindecomp} are negligible.
In order to derive
\begin{align}\label{Restslt}
n^{\frac 1 {(\alpha-k)\beta}}\sum_{r=k}^n R_r^n \toop 0,
\end{align}
we may argue along the lines of the proof of (5.22) in \cite[Proposition 5.2]{BasLacPod2016} where this statement is derived in the context of power variation (note that $R_r^n$ corresponds to $R_r^{n,0}$ in their notation). Key to the proof is the estimate \cite[Lemma 5.7]{BasLacPod2016}, which we generalize to our setting in Lemma~\ref{lemthetaest}.  Similarly, we obtain  
\begin{align}\label{Questslt}
n^{\frac 1 {(\alpha-k)\beta}}\sum_{r=k}^n (Q_r^n-Z_r) \toop 0
\end{align}
by arguing along the lines of the proof of (5.24) in \cite[Proposition 5.2]{BasLacPod2016}. The proof relies on the estimates \eqref{StaVarEst}--\eqref{scaparjest} and \cite[Eq.~(5.15), (5.18)]{BasLacPod2016}, as well as on Lemma~\ref{lemEstslt}. 
The estimate \cite[Eq.~(5.15)]{BasLacPod2016} is in our context replaced by \eqref{klsjhdfkhsdklfhhd}, where we need to argue that for sufficiently large $N$ the set $\{\rho_{j}^n\,:\, n\in\{N,...,\infty\},j\in\N\}$ is bounded away from 0 and $\infty$, which is done in Lemma~\ref{lemrhobou}. 

It therefore remains to show that $Z_r$ is in the domain of attraction of a $(k-\alpha)\beta$-stable random variable, which we do in two steps. First we define the random variable\vspace{-1ex}
\[Q:=\ol\Phi(L_{k+1}-L_{k})-\E[\ol\Phi(L_{k+1}-L_{k})],\quad\text{where }\quad \ol \Phi(x):=\sum_{j=1}^\infty \Phi_{\rho_j^\infty}(\phi_j^\infty(0)x)\]
and show that it is in the domain of attraction of a $(k-\alpha)\beta$-stable random variable $S$.
Thereafter we argue that we can find $r>(k-\alpha)\beta$ such that 
\begin{align}\label{ZQDifslt}
\P(|Z_k-Q|>x)\leq Cx^{-r},\quad \text{for all }x\geq 1,
\end{align}
which yields  that $Z_k$ is in the domain of attraction of $S$ as well, and an application of \cite[Theorem 1.8.1]{SamTaq1994} concludes the proof. 

Let us first remark that the function $\ol\Phi$ and the random variable $Q$ are well-defined. Indeed, since $\rho_j^\infty\to\rho^\infty$, the set $\{\rho^\infty_j\}_{j\in\N}$ is bounded away from 0 and $\infty$ and by \eqref{sljsdfljs} it follows for any $\gamma\in (p,\beta)$ that 
\begin{equation}
 \sum_{j=1}^\infty |\Phi_{\rho_j^\infty}(\phi_j^\infty(0)x)|\leq C \sum_{j=1}^\infty |\phi_j^\infty(0)x|^\gamma\leq C |x|^\gamma\sum_{j=1}^\infty j^{\gamma(\alpha-k)}.
\end{equation}
By choosing $\gamma>1/(k-\alpha)$ it follows that $\ol\Phi$ and $Q$ are well-defined. 
Moreover, an application of the dominated convergence theorem shows that $\ol\Phi$ is continuous.
In order to show that $Q$ is in the domain of attraction of a $(k-\alpha)\beta$-stable random variable we next determine constants $c_-,c_+$ such that
\begin{equation}\label{slfjsdhfsdfs}
 \lim_{x\to \infty} x^{(k-\alpha)\beta}\P(Q<-x)=c_-,\quad \lim_{x\to \infty} x^{(k-\alpha)\beta}\P(Q>x)=c_+.
\end{equation}
From \eqref{slfjsdhfsdfs} it follows by \cite[Theorem 1.8.1]{SamTaq1994} that $Q$ is in the domain of attraction of a $(k-\alpha)\beta$-stable with scale parameter $\rho_1$ and skewness parameter $\eta_1$, given by
\begin{align}\label{ScaSkeParslt}
\rho_1:= \bigg(\frac{c_++c_-}{\tau_{(k-\alpha)\beta}}\bigg)^{1/(k-\alpha)\beta},\qquad\text{and}\qquad\eta_1:=\frac{c_+-c_-}{c_++c_-}.
\end{align}
Here the constant $\tau_\gamma$,  $\gamma\in(0,2)$, is  defined as
\begin{align}\label{tauDefslt}
\tau_\gamma:=
\begin{cases}
\frac {1-\gamma}{\Gamma(2-\gamma) \cos(\pi\gamma/2)}&\text{if }\gamma\neq 1,\\
\pi/2 & \text{if }\gamma=1.
\end{cases}
\end{align}
In the following we derive explicit expressions for $c_+$ and $c_-$, which are stated in \eqref{TaiProPosslt} and \eqref{TaiProNegslt} below.
For $x>0$ it holds by substituting $t=(x/u)^{1/(k-\alpha)}$ that
\begin{align}
x^{1/(\alpha-k)}\ol \Phi(x)&=x^{1/(\alpha-k)}\int_0^\infty\Phi_{\rho_{1+[t]}^\infty}(\phi^\infty_{1+[t]}(0)x)\,dt\nonumber\\
&=\frac 1 {k-\alpha} \int_0^\infty\Phi_{\rho_{1+[(x/u)^{1/(k-\alpha)}]}^\infty}\big(\phi^\infty_{1+[(x/u)^{1/(k-\alpha)}]}(0)x\big)u^{-1+1/(\alpha-k)}\,du\qquad \label{dsklfjslkdghfgsfd}\\
&\to\frac 1 {k-\alpha} \int_0^\infty\Phi_{\rho^\infty}(k_\alpha u)u^{-1+1/(\alpha-k)}\,du := \kappa_+,\quad \text{as $x\to\infty$},\label{kappa+slt}
\end{align}
where $k_\alpha=\alpha(\alpha-1)\dots(\alpha-k+1).$ The convergence as well as the existence of the integral follow from the estimate \eqref{sljsdfljs} and the dominated convergence theorem, where we use that $\{\rho_j^\infty\}$ is bounded away from 0 and $\infty$.
The convergence of the integrand from \eqref{dsklfjslkdghfgsfd} as $x\to\infty$ follows since by the mean value theorem for all $t\in\R$ there is a $\xi_t\in[t-k-1,t]$ such that 
\[\phi^\infty_{[t]}(0)=h_k([t])=k_\alpha (\xi_t)_+^{\alpha-k},\]
which implies the convergence
\[\phi^\infty_{1+[(x/u)^{1/(k-\alpha)}]}(0)x\to k_\alpha u,\quad\text{as $x\to\infty.$}\]
Similarly we obtain for $x<0$ that
\begin{align}\label{kappa-slt}
|x|^{1/(\alpha-k)}\ol \Phi(x)
&\to\frac 1 {k-\alpha} \int_{-\infty}^0\Phi_{\rho^\infty}(k_\alpha u)|u|^{-1+1/(\alpha-k)}\,du := \kappa_-,\quad \text{as $x\to-\infty$}.
\end{align}
We argue next that
\begin{align}\label{TaiProPosslt}
\lim_{x\to\infty} x^{(k-\alpha)\beta}\P(Q>x)&=\tau_\beta\rho_L\big(\kappa^{k-\alpha}_+ \mathds 1_{\{\kappa_+>0\}}+\kappa^{k-\alpha}_- \mathds 1_{\{\kappa_->0\}}\big):=c_+,
\end{align}
where $\tau_\beta$ was defined in \eqref{tauDefslt} and $\rho_L$ denotes the scale parameter 
of the L\'evy process $L$. To this end we make the decomposition
\begin{align}\label{ProbabilityDecompPosNeg}
\P(Q>x)= \P(Q>x,L_{k+1}-L_k>0)+\P(Q>x,L_{k+1}-L_k<0),
\end{align}
and analyse the two summands separately. Consider the first summand and assume $\kappa_+>0$. By \eqref{kappa+slt} it follows that $\ol\Phi(y)\to\infty$ as $y\to\infty$ and we have for sufficiently large $x$ that
\[\P(\ol\Phi(L_{k+1}-L_{k})>x,L_{k+1}-L_k>0)=\P(|\ol\Phi(L_{k+1}-L_{k})|>x,L_{k+1}-L_k>0).\]
Applying Lemma \ref{lemAsyEquSta} with $\xi(x)=\ol\Phi(x)$ and $\psi(x)=x^{1/(k-\alpha)}\kappa_+$, we deduce from \eqref{kappa+slt} that
\begin{align*}
\lim_{x\to\infty} x^{(k-\alpha)\beta}\P(Q>x,L_{k+1}-L_{k}>0)&=\lim_{x\to\infty} x^{(k-\alpha)\beta}\P\big(\kappa_+^{k-\alpha}(L_{k+1}-L_{k}) >x^{k-\alpha}\big)\\
&=\tau_\beta \rho_L^\beta\kappa_+^{(k-\alpha)\beta},
\end{align*}
where the second identity follows from \cite[Property 1.2.15]{SamTaq1994}.
If $\kappa_+<0,$ it follows from \eqref{kappa+slt} that $\limsup_{x\to\infty}\ol\Phi(x)\leq 0$ and therefore that $\ol\Phi(x)$ is bounded for $x\geq 0$. We obtain 
\[\lim_{x\to\infty} x^{(k-\alpha)\beta}\P(Q>x,L_{k+1}-L_{k}>0)=0.\]
The same identity holds for $\kappa_+=0$, as follows from Lemma \ref{lemAsyEquSta}, \eqref{kappa+slt}, and the estimate
\[\P(\ol\Phi(L_{k+1}-L_{k})>x,L_{k+1}-L_k>0)\leq \P(|\ol\Phi(L_{k+1}-L_{k})|>x, L_{k+1}-L_k>0).\]
We conclude that the first summand of \eqref{ProbabilityDecompPosNeg} satisfies
\begin{align*}
\lim_{x\to\infty} x^{(k-\alpha)\beta}\P(Q>x,L_{k+1}-L_{k}>0)&=\tau_\beta\rho_L \kappa^{k-\alpha}_+ \mathds 1_{\{\kappa_+>0\}}.
\end{align*}
By similar arguments, applying Lemma~\ref{lemAsyEquSta} on the function $\xi(x)= \ol \Phi(-x)$ and using \eqref{kappa-slt}, we obtain for the second summand of \eqref{ProbabilityDecompPosNeg} the convergence
\begin{align*}
\lim_{x\to\infty} x^{(k-\alpha)\beta}\P(Q>x,L_{k+1}-L_{k}<0)&=\tau_\beta\rho_L \kappa^{k-\alpha}_- \mathds 1_{\{\kappa_->0\}},
\end{align*}
which completes the proof of \eqref{TaiProPosslt}. Arguing similarly for $\P(Q<-x)$ we derive that
\begin{align}\label{TaiProNegslt}
\lim_{x\to\infty} x^{(k-\alpha)\beta}\P(Q<-x)&=\tau_\beta\rho_L\big(|\kappa_+|^{k-\alpha} \mathds 1_{\{\kappa_+<0\}}+|\kappa_-|^{k-\alpha} \mathds 1_{\{\kappa_-<0\}}\big) := c_-.
\end{align}
This shows that $Q$ is in the domain of attraction of a  $(k-\alpha)\beta$-stable random variable with location parameter 0, and scale and skewness parameters as given in \eqref{ScaSkeParslt}.

Now the proof of the theorem is completed by showing \eqref{ZQDifslt}. To this end it is by Markov's inequality sufficient to show that $\E[|Z_k-Q|^r]<\infty$ for some $r>(k-\alpha)\beta.$
Since $(k-\alpha)\beta>1$ an application of Minkowski's inequality yields
\begin{align}\label{LrDisZQslt}
\|Z_k-Q\|_r\leq \sum_{j=1}^\infty\big\|\Phi_{\rho_j^\infty}(U^\infty_{j+k,k})-\Phi_{\rho_j^\infty}\big(\phi_{j}^\infty (0)(L_{k+1}-L_{k})\big)\big\|_r.
\end{align}
We remark that by the mean value theorem there exists a constant $C>0$ such that for all $x\in[0,1]$ and $j\in\N$ it holds that
\[|\phi_{j+k}^\infty(x)-\phi_{j}^\infty(0)|=|h_k(j+k+x)-h_k(j)|\leq Cj^{\alpha-k-1}.\]
Since $\{\rho_j^\infty\}_{j\in\N}$ is bounded away from 0, there is a $\delta>0$ with $\delta<\rho_j^\infty$ for all $j$. Letting $r_\eps=(k-\alpha)\beta+\eps$ with $\eps\in(0,\delta)$, an application of Lemma \ref{lemEstslt} yields
\begin{align}
&\big\|\Phi_{\rho_j^\infty}(U^\infty_{j+k,k})-\Phi_{\rho_j^\infty}(\phi_{j}^\infty (0)(L_{k+1}-L_{k}))\big\|_{r_\eps}\\
&\leq C\big(\|\phi_{j+k}^\infty-\phi_j^\infty(0)\|^{1-\eps}_{L^\beta([0,1])}+\|\phi_{j+k}^\infty-\phi_j^\infty(0)\|^{\frac{1}{k-\alpha+\eps/\beta}}_{L^\beta([0,1])}\big)\leq C (j^{(\alpha-k-1)(1-\eps)}+j^{\frac{\alpha-k-1}{k-\alpha+\eps/\beta}}).\label{dslfjsdlfj}
\end{align} 
For sufficiently small $\eps>0$, both powers of $j$ on the right-hand side of \eqref{dslfjsdlfj} are smaller than $-1$, which together with \eqref{LrDisZQslt} implies $\|Z_k-Q\|_r<\infty$, and thus \eqref{ZQDifslt}. Since $Q$ is in the domain of attraction of a $(k-\alpha)\beta$-stable random variable with scale parameter $\rho_1$ and skewness parameter $\eta_1$, and $r>(k-\alpha)\beta$, so is $Z_k$. This completes the proof of Theorem \ref{the2ndGenVar}(ii).\qed

\section{Auxiliary results}\label{lsjdfghsdfoghs}


In this section we show some technical results used in the proofs of Theorems~\ref{lsjdlfjssss} and \ref{the2ndGenVar}.

\begin{lem}\label{DerBou}
 Let $p$ be as in \eqref{Est2}. For any $\eps>0$ there exists a finite constant $C_\eps>0$ such that for all $\rho\in[ \eps,\eps^{-1}]$,  $a\in\R$ and  $x,y>0$ we have that
 \[F(a,x,y):=\bigg|\int_{0}^y\int_0^x \Phi''_\rho(a+u+v)\diff u\diff v\bigg|\leq C(x^p\wedge x)(y^p\wedge y).\]
 \end{lem}
 \begin{proof}
Let us first remark that $x^p\wedge x= x\mathds 1_{\{x\leq 1\}}+x^p\mathds 1_{\{x>1\}}$ since $p<1$.
By assumption, $\Phi_\rho'(x)$ and $\Phi_\rho''(x)$ are uniformly bounded for $\rho\in [\eps,\eps^{-1}]$ and $x\in\R$.
Boundedness of $\Phi_\rho''$ immediately implies $F(a,x,y)\leq Cxy.$ Moreover, it holds that
\begin{align*}
\int_{0}^y \int_0^x \Phi''_\rho(a+u+v)\diff u\diff v&= \int_{0}^y  \Phi'_\rho(a+x+v)-\Phi'_\rho(a+v)\diff v\\
&=\big(\Phi_\rho(a+x+y)-\Phi_\rho(a+y)\big)-\big(\Phi_\rho(a+x)-\Phi_\rho(a)\big).
\end{align*}
The first equality and boundedness of $\Phi'_\rho$ implies $F(a,x,y)\leq Cy,$ and consequently $F(a,x,y)\mathds 1_{\{x>1\}}\leq Cx^py\mathds 1_{\{x>1\}},$
 and similarly $F(a,x,y)\mathds 1_{\{y>1\}}\leq Cxy^p\mathds 1_{\{y>1\}}.$ Finally, the second equality together with \eqref{Est2} implies that  $F(a,x,y)\mathds 1_{\{x>1,y>1\}}\leq C x^p \mathds 1_{\{x>1,y>1\}}\leq C x^p y^p\mathds 1_{\{x>1,y>1\}},$ completing the proof.
\end{proof}

\begin{lem}\label{lemthetaest}
For all $\gamma\in[1,2]$ there exists a $C>0$ such that for all $n\in\N$, $r\in\{k,\dots, n\}$, $j\in\N$ and $l\geq j$ it holds that
\[\E[|\vartheta_{r,j,l}^n|^\gamma]\leq 
\begin{cases}
Cj^{(\alpha-k)\beta} l^{(\alpha-k)\beta} \qquad & \text{for }\beta<\gamma<\beta/p,\\
Cj^{(\alpha-k)\gamma} l^{(\alpha-k)\gamma} & \text{for }\gamma<\beta, 
\end{cases}
\]
where $\vartheta_{r,j,l}^n$ was defined in \eqref{thetadefslt}.
\end{lem}
\begin{proof}
It is sufficient to consider the case $r=1$, since for fixed $j,l,n$ the sequence $(\vartheta_{r,j,l}^n)_{r\in\N}$ is stationary.
Without loss of generality we may assume that $l\geq 2\vee j$ since the case $l=j=1$ can be covered by choosing a larger constant. To this end we remark that $(\E[|\vartheta_{1,1,1}^n|^\gamma])_{n\in\N}$ is bounded, since $Y^n_r\sim$ \sbssp{\rho^n} with $\rho^n$ (which was introduced in \eqref{rhoDef}) bounded away from 0 and $\infty$ by \cite[Lemma 5.3]{BasLacPod2016}. By definition of $\vartheta$ it holds that
\begin{align*}
\vartheta_{1,j,l}^n 
=&\ \E[f(Y^n_1)\,|\,\G_{1-j}^1\vee \G_{1-l}]-\E[f(Y^n_1)\,|\, \G_{1-l}]\\
&-\big\{\E[f(Y^n_1)\,|\,\G_{1-j}^1\vee \G_{-l}]-\E[f(Y^n_1)\,|\, \G_{-l}]\big\}.
\end{align*}
Define for $-\infty\leq a<b\leq 1$ the random variable 
\[U_{[a,b]}^n=\int_{a}^{b}\phi_1^n(s)\diff L_s.\]
Let in the following $\wt L$ be an independent copy of $L$ and define $\wt U_{[a,b]}^n$ accordingly, and denote by $\wt\E$ the expectation with respect to $\wt L$ only. 
Moreover, we denote by $\rho_{j,l}^n=\|\phi_1^n\|_{L^\beta([1-l,1-j]\cup[2-j,1])}$, i.e.\ the scale parameter of $\int_{1-l}^{1-j} \phi_1^n\diff L_s +\int_{2-j}^{1} \phi_1^n\diff L_s.$
Then, decomposing $\int_{-\infty}^1 \phi_1^n \diff L_s$ into the independent integrals 
\[\int_{-\infty}^1 \phi_1^n \diff L_s=\int_{-\infty}^{-l}\phi_1^n \diff L_s +\int_{-l}^{1-l}\phi_1^n \diff L_s+\int_{1-l}^{1-j}\phi_1^n \diff L_s+\int_{1-j}^{2-j}\phi_1^n \diff L_s+\int_{2-j}^{1}\phi_1^n \diff L_s\]
we obtain the expression
\begin{align*}
 \vartheta_{1,j,l}^n = {}& \wt \E \Big[ \Phi_{\rho_{j,l}^n}(U^n_{[-\infty,-l]}+U^n_{[-l,1-l]}+U^n_{[1-j,2-j]})-\Phi_{\rho_{j,l}^n}(U^n_{[-\infty,-l]}+U^n_{[-l,1-l]}+\wt U^n_{[1-j,2-j]})\\
{}& 
-\Phi_{\rho_{j,l}^n}(U^n_{[-\infty,-l]}+\wt U^n_{[-l,1-l]}+U^n_{[1-j,2-j]})+\Phi_{\rho_{j,l}^n}(U^n_{[-\infty,-l]}+\wt U^n_{[-l,1-l]}+\wt U^n_{[1-j,2-j]})\Big]\\
={}& \wt \E \Big[ \int_{\wt U^n_{[-l,1-l]}}^{ U^n_{[-l,1-l]}}\int_{\wt U^n_{[1-j,2-j]}}^{ U^n_{[1-j,2-j]}}\Phi''_{\rho_{j,l}^n}(U^n_{[-\infty,-l]}+u+v)\diff u \diff v\Big],
\end{align*}
and by substitution there is a random variable $\wt W^n_{j,l}$ such that
\[|\vartheta_{1,j,l}^n|\leq\wt \E \bigg[ \bigg |\int_{0}^{|\wt U^n_{[-l,1-l]}- U^n_{[-l,1-l]}|}\int_{0}^{ |\wt U^n_{[1-j,2-j]}-U^n_{[1-j,2-j]}|}\Phi''_{\rho_{j,l}^n}(\wt W^n_{j,l}+u+v)\diff u \diff v\bigg|\bigg].
\]
We denote $\varphi_p(x):=|x|^p\wedge |x|$. Suppose in the following that $\gamma>\beta$. Using Lemma~\ref{DerBou}, Jensen's inequality, the inequality $\varphi_p(|x-y|)\leq 2 (\varphi_p(|x|)+\varphi_p(|y|))$ and the independence of $U$ and $\wt U$, we obtain that
\begin{align*}
&\E[|\vartheta_{1,j,l}^n|^\gamma]\leq C\E[\wt \E[\varphi_p^\gamma( |\wt U^n_{[-l,1-l]}- U^n_{[-l,1-l]}|)\varphi_p^\gamma(|\wt U^n_{[1-j,2-j]}-U^n_{[1-j,2-j]}|) ]]\\
&\quad \leq C\E[\wt \E[\varphi_p^\gamma( |\wt U^n_{[-l,1-l]}|)+\varphi_p^\gamma( |U^n_{[-l,1-l]}|)]]
\E[\wt \E[\varphi_p^\gamma(|\wt U^n_{[1-j,2-j]}|)+\varphi_p^\gamma( |U^n_{[1-j,2-j]}|)]]\\
&\quad \leq C \|\phi_1^n\|^\beta_{L^\beta([-l,1-l])}\|\phi_1^n\|^\beta_{L^\beta([1-j,2-j])}\leq C l^{(\alpha-k)\beta}j^{(\alpha-k)\beta}.
\end{align*}
 In the third inequality we used the estimate \eqref{StaVarEst}, where we remark that by assumption $\gamma>\beta$ and $p\gamma<\beta$, and the expression \eqref{ScaParInt} for the scale parameter of integrals with respect to a stable L\'evy process. The last inequality follows from \eqref{scaparest}.
 For $\gamma<\beta$ we use the  same arguments above, however, due to the fact that \eqref{StaVarEst} gives at different estimate in this case we obtain the bound $\E[|\vartheta_{1,j,l}^n|^\gamma]\leq Cl^{(\alpha-k)\gamma} j^{(\alpha-k)\gamma}$, which concludes the proof.  
 \end{proof}

\begin{lem}\label{lemrhobou}
The set $\{\rho_{j}^n\,:\, n\in\{N,...,\infty\},j\in\N\}$ is bounded away from 0 and $\infty$ for sufficiently large $N\in\N.$
\end{lem}
\begin{proof}
Choose $\eps>0$ such that $\eps<\rho^\infty<\eps^{-1}$ and $\eps<\rho_j^\infty<\eps^{-1}$ for all $j\in\N$. It follows from \cite[Lemma 5.3]{BasLacPod2016} that $\rho^n\to\rho^\infty$ and we can choose $N$ sufficiently large such that $|\rho^n-\rho^\infty|<\eps/3$ for all $n>N$, implying that $2\eps/3<\rho^n<\eps^{-1}+\eps/3$. Moreover, $\rho_j^n$ converges to $\rho^n$ uniformly in $n$ by the estimate
\[|(\rho_j^n)^\beta-(\rho^n)^\beta|=\|\phi_j^n\|^\beta_{L^\beta([0,1])}\leq C j^{\beta(\alpha-k)},\]
 where we used \eqref{scaparest}, and that the function $x\mapsto|x|^\beta$, restricted to a compact set, is uniformly continuous.
Consequently, we can find a $J>0$ such that for all $j>J$ and all $n$ it holds that $|\rho_j^n-\rho^n|<\eps/3$, implying that $\eps/3<\rho_j^n<\eps^{-1}+2\eps/3$ for all $j>J$ and $n>N$. For $j\in\{1,...,J\}$ we use that $\rho_j^n\to \rho_j^\infty\in(\eps,\eps^{-1})$ as $n\to\infty$, which follows similarly from \eqref{scaparjest}. Therefore, choosing $N$ larger if necessary, we obtain $\eps/3<\rho_j^n<\eps^{-1}+1$ for all $j\in\N$ and $n>N$.
\end{proof}

The following auxiliary result was derived in \cite{BasLacPod2016} in the context of power variation. The proof relies only the estimate \eqref{klsjhdfkhsdklfhhd} on $\Phi_\rho$. 
\begin{lem}(\cite[Lemma 5.4]{BasLacPod2016}).\label{lemEstslt}
Under the setting of Theorem~\ref{the2ndGenVar}(ii), we have for any $q\geq 1$ with $q\neq\beta$ that there exists $\delta>0$ and a finite constant $C$ such that for all $\eps\in(0,\delta)$, $\rho>\delta$ and $\kappa,\tau\in L^\beta([0,1])$ satisfying  $\|\kappa\|_{L^\beta([0,1])}, \|\tau\|_{L^\beta([0,1])} \leq  1$ and 
\begin{align*}
&\bigg\|\Phi_\rho\bigg(\int_0^1 \kappa(s) \,dL_s\bigg)-\Phi_\rho\bigg(\int_0^1 \tau(s) \,dL_s\bigg)\bigg\|_q\\
&\leq
\begin{cases}
C \|\kappa-\tau\|^{\beta/q}_{L^{\beta}([0,1])}	& \beta<q\\
C\Big\{\big(\|\kappa\|_{L^{\beta}([0,1])}^{(\beta-q)/q -\eps}+\|\tau\|_{L^{\beta}([0,1])}^{(\beta-q)/q -\eps}\big)\|\kappa-\tau\|^{1-\eps}_{L^\beta([0,1])}+ \|\kappa-\tau\|^{\beta/q}_{L^\beta([0,1])}\Big\}&\beta>q.
\end{cases}
\end{align*}
\end{lem}

We will need the following minor extension of \cite[Lemma~2.1]{PipTaq2003}:

\begin{lem}(\cite[Lemma~2.1]{PipTaq2003}).\label{dsklfhdskfhg}
 Let $\{X_n:n\in \N_0\}$  denote symmetric $\beta$-stable random variables such that $X_n\to X_0$ in probability. Suppose that $f:\R\to \R$ is a measurable function such that $\E[ |f(X_0)|^p]<\infty$ for some $p\geq 1$.  Then, $\E[|f(X_n)-f(X_0)|^p]\to 0$. 
\end{lem}

Note that Lemma~\ref{dsklfhdskfhg} relies heavily on the $\beta$-stable assumption, and a similar result (with no continuity assumptions on $f$) 
does not hold for e.g.\ discrete random variables. 

\begin{proof}
If $f$ is bounded, $p=2$ and $X_n\to X_0$ almost surely, Lemma~\ref{dsklfhdskfhg} is  \cite[Lemma~2.1]{PipTaq2003}.
However, going through the proof of \cite[Lemma~2.1]{PipTaq2003} shows that it also  holds for a general $p\geq 1$ and if $X_n\to X$ in probability, by using the same arguments.  To extend Lemma~\ref{dsklfhdskfhg} from bounded  $f$,  to unbounded $f$ satisfying  $\E[|f(X)|^p]<\infty$, it is enough to show tightness of $\{|f(X_n)|^p:n\geq 1\}$, due to a truncation argument. The density of $X_n$ satisfies 
\begin{equation}\label{dsfljsdf}
 f_{X_n}(x)= \rho_n^{-1} g_\beta(x/\rho_n), \qquad x\in \R, 
\end{equation}
where $g_\beta$ is the density of a standard symmetric $\beta$-stable random variable and $\rho_n$ is the scale parameter for $X_n$ for $n\in \N_0$.   Since $\E[|f(X_0)|^p]<\infty$ and  $\rho_n\to \rho$ (follows since $X_n\to X$ in distribution), we deduce tightness of $\{|f(X_n)|^p:n\geq 1\}$ from \eqref{dsfljsdf}. This completes the  proof. 
\end{proof}

\begin{lem}\label{lemAsyEquSta}
Let $\psi,\xi$ be continuous functions on $\R$ with $\psi(x)\sim \xi(x)$ for $x\to\infty$. Let $X$ be a random variable taking values in $\R_+$ and $\gamma\geq 0$ such that
\[\lim_{x\to\infty}x^\gamma\P(|\psi(X)|>x)=\kappa\]
where $\kappa\in[0,\infty).$
Then it holds that
\[\lim_{x\to\infty}x^\gamma\P(|\xi(X)|>x)=\kappa.\]
\end{lem}
\begin{proof}
Denote $\psi(x)=\xi(x)\varphi(x)$ with $\varphi(x)\to 1$ for $x\to\infty$. Let $\eps>0$. By continuity of $\psi$ and $\xi$ we can choose $x$ sufficiently large such that $\varphi(y)\in(1-\eps,1+\eps)$ whenever $\min(|\psi(y)|,|\xi(y)|)>x$ and $y\geq 0$. Since $X$ takes values in $\R_+$, this implies that $\varphi(X)\in(1-\eps,1+\eps)$ whenever $|\psi(X)|>x$ or $|\xi(X)|>x.$
It follows that
\begin{align*}
& x^\gamma|\P(|\psi(X)|>x)-\P(|\xi(X)|>x)|
=\E\big[x^\gamma \big(\mathds 1_{\{|\psi(X)|>x>|\xi(X)| \} }+\mathds 1_{ \{ |\psi(X)|<x<|\xi(X)| \}}\big)\big]\\
&\qquad \leq 2\E\big[x^\gamma \mathds 1_{\{\frac{x}{1+\eps}<|\psi(X)|<\frac{x}{1-\eps}\} }\big]
= 2\E\big[x^\gamma \mathds 1_{\{\frac{x}{1+\eps}<|\psi(X)|\} }-x^\gamma\mathds 1_{\{\frac{x}{1-\eps}\leq |\psi(X)|\} }\big]\\
&\qquad \to 2\kappa((1+\eps)^\gamma-(1-\eps)^\gamma),\quad\text{as $x\to\infty$.}
\end{align*}
 The lemma follows by letting $\eps\to 0.$
\end{proof}

\begin{proof}[Proof of Remark~\ref{examp}]

(i): We will start by verifying (B) for any bounded measurable function $f$.  Let $g_{\beta}$ denote the density of a  standard symmetric $\beta$-stable random variable. By substitution we have  
\begin{equation}\label{sdfljsdf}
\Phi_\rho(x) = \int_\R f(y) g_{\beta}((y-x)/\rho)\,dy-\int_\R f(y) g_{\beta}(y/\rho)\,dy.
\end{equation}
Recall from \eqref{sdfsdf} that  
  $g_\beta\in C^\infty(\R)$, and for all $r\geq 1$,  the $r$th derivative of $g_\beta$ satisfies
\begin{equation}\label{sdfsdf-2}
|g_\beta^{(r)}(x)|\leq C (1\wedge |x|^{-1-\beta-r}), \qquad x\in \R. 
\end{equation}
By the \eqref{sdfljsdf}, \eqref{sdfsdf-2} and using that $f$ is bounded, it follows that $\rho\mapsto \Phi_\rho(x)$ is $C^1((0,\infty))$ and  
\begin{align}
\frac{\partial }{\partial \rho} \Phi_\rho(x) = {}& -\rho^{-2}\Big( \int_\R \Big(f(y)g_\beta'((y-x)/\rho) (y-x)\Big)\,dy-\int_\R \Big(f(y)g_\beta'(y/\rho) y\Big)\,dy\Big)\\ = {}& 
- \int_\R \Big(f(x+\rho y )g_\beta'(y) y\Big)\,dy+\int_\R \Big(f(\rho y )g_\beta'(y) y\Big)\,dy,
\end{align}
which implies  existence of  $C>0$ such that  $|\frac{\partial }{\partial \rho} \Phi_\rho(x)|\leq C $ for all $\rho\in [\epsilon,\epsilon^{-1}]$ and $x\in \R$. By similar arguments one can verify the remaining conditions of (B). 


(ii):  Next we suppose that   $f\in L^1_{\textrm{loc}}(\R)$ and there exists  $K>0$ and $q\leq 1$ such that   $f\in C^3([-K,K]^c)$ and $|f'(x)|, |f''(x)|, |f'''(x)|\leq C$ and $|f'(x)|\leq C |x|^{q-1}$   for $|x|>K$.  In the following we will verify that  $f$ satisfies (B) with $p=q$ when $q>0$, and $p=0$ when $q<0$. Let  $\xi\in C^\infty_c(\R)$ be a function such that $\xi =1$ on  $[-K,K]$.
By the equality $1=\xi+(1-\xi)$ and substitution we have 
\begin{align}
 &\Phi_\rho(x) -\int f(\rho y)g_\beta(y)\,dy = \int f(x+\rho y) g_\beta(y)\,dy  \\ 
 {}& \qquad =\int f(y)\xi(y)g_\beta((y-x)/\rho)\,dy+\int f(x+y\rho )\big(1-\xi(x+y\rho )\big)g_\beta(y)\,dy \\
& \qquad =: \bar \Phi_\rho(x)+\tilde \Phi_\rho(x). 
\end{align}
Since $f$ is locally integrable and $\xi$ has compact support we have $f\xi\in L^1(\R)$, and due to the fact that  $|g'|$ is bounded 
\begin{align}\label{sdlfjsdlj}
\Big|\frac{\partial}{\partial x}\bar\Phi_\rho(x)\Big| = \Big|\rho^{-1}\int f(y)\xi(y)g_\beta'((y-x)/\rho)\,dy\Big|\leq C\rho^{-1} \int |f(s)\xi(s)|\,ds<\infty. 
\end{align}
On the other hand, it follows that $(f(1-\xi))'$ is bounded. Indeed, since  $f(1-\xi)=0$ on  $[-K,K]$ it is enough to show that $(f(1-\xi))'$ is bounded for $|x|>K$. For $|x|>K$  we have $(f(1-\xi))'=f' (1-\xi)-f\xi'$ which is bounded due to the fact that $f'$ is bounded and $f$ is continuous for $|x|>K$, and $\xi'$ has compact support.  Therefore, 
\begin{align}\label{lsjdfljsdghsgs}
\Big|\frac{\partial}{\partial x}\tilde \Phi_\rho(x)\Big| = \Big| \int (f(1-\xi))'(x+y\rho) g_\beta(y)\,dy\Big| \leq C \int |g_\beta(y)|\,dy<\infty.
\end{align}
From \eqref{sdlfjsdlj} and \eqref{lsjdfljsdghsgs} it follows that $\frac{\partial}{\partial x}\Phi_\rho(x)$ is bounded. By similar arguments one can verify the remaining conditions of \eqref{Est3}. To verify \eqref{Est2} we will use that $g_\beta$ is both Lipschitz continuous and bounded,  and hence  for any $p\in [0,1],$ $\rho\in[\epsilon,\epsilon^{-1}]$
\begin{align}
|\bar \Phi_\rho(x)-\bar\Phi_\rho(y)| \leq{}& \int |f(u)\xi(u)(g_\beta((u-x)/\rho)-g_\beta((u-y)/\rho)|\,du 
\\ \leq {}& C\Big(1\wedge |x-y|\Big)\int |f(u)\xi(u)|\,du\leq C |x-y|^p.
\end{align} 
For $0<q\leq 1$ and  $x\neq 0$ we have that 
$|(f(1-\xi))'(x)|\leq C |x|^{q-1}$ which implies that $f(1-\xi)$ is $q$-H\"older continuous, and therefore
\begin{align}
|\tilde \Phi_\rho(x)-\tilde \Phi_\rho(y)|\leq {}&  \int \Big|(f(1-\xi))(x+u) - (f(1-\xi))(y+u))\Big|g_\beta(u)\,du
\\ \leq {}& C\Big( \int g_\beta(u)\,du \Big) |x-y|^q. 
\end{align}
This concludes the proof of \eqref{Est2} with $p=q$ when  $0<q\leq 1$. 
For $q< 0$, we have that $f\in L^1(\R)$, and hence it follows by \eqref{sdfljsdf} and boundedness of $g_\beta$ that 
$|\Phi_\rho(x)|\leq C$ for all $x\in \R$ and $\rho\in [\epsilon,\epsilon^{-1}]$, which shows \eqref{Est2} with $p=0$. %
\end{proof}

\begin{remark}\label{lsdjflsdhfh}
In the following we proof the statements on the Appell rank at the begining of Subsection~\ref{sec2.2}.
Suppose first that  $f$ is an even function. Since $S$ is a symmetric random variable and $\Phi_\rho(x)= \E[ f(x+\rho S)]-\E[f(\rho S)]$, we have that $x\mapsto \Phi_\rho(x)$ is an even function for all $\rho$.  Hence, $\frac{\partial }{\partial x} \Phi_\rho(0)=0$ and $\frac{\partial^2 }{\partial x\partial \rho} \Phi_\rho(0)=0$.
Next  consider the function $f(x)=\sin(ux)$ for all $x\in \R$, where $u\neq 0$.  We have that 
\begin{equation}
\Phi_\rho(x) = \E[ \sin(u(x+\rho S))]- \E[ \sin(u\rho S)]= \Im\Big(  \E[ e^{i u (x+\rho S)}]\Big)= \sin(ux) e^{-|\rho u|^\beta},
\end{equation}
and hence $\frac{\partial }{\partial x}\Phi_\rho(0)=u e^{-|\rho u|^\beta} \neq 0$. Finally, we  let $f(x) = \1_{(-\infty,u]}(x)$ for all $x\in \R$,  where  $u\in \R$. Then 
\begin{equation}
\Phi_\rho(x) = \P(S\leq (u-x)/\rho)- \P(S\leq u/\rho),
\end{equation}
and hence $\frac{\partial }{\partial x}\Phi_\rho(0)=-\rho g_\beta(u/\rho)$, where $g_\beta$ denotes  the density of a standard $S\beta S$ random variable. Since $g_\beta(x) \neq 0$  for all $x\in \R$ (see e.g.\ Theorem~1.2 in \cite{Wat07}), it follows that $\frac{\partial }{\partial x}\Phi_\rho(0)\neq 0$, which completes the proofs of the statements. 
\end{remark}

\subsection*{Acknowledgment}
Claudio Heinrich and Mark Podolskij acknowledge financial support from the project 
``Ambit fields: probabilistic properties and statistical inference'' funded by Villum Fonden 
and from CREATES funded by the Danish
National Research Foundation. 
Claudio Heinrich acknowledges financial support from project number 88511 founded by the Volkswagen Foundation. Andreas Basse-O'Connor acknowledge financial support by the Grant DFF--4002-00003 funded by the Danish Council for Independent Research

\bibliographystyle{chicago}

\end{document}